\documentclass[a4paper,12pt,intlimits,oneside,reqno]{amsart}
\usepackage{enumerate}
\usepackage{amsfonts,amsmath, amsxtra,amssymb,latexsym, amscd,amsthm}

\usepackage{latexsym,amssymb}
\newtheorem{theorem}{Theorem}[section]

\newtheorem{lemma}[theorem]{Lemma}

\theoremstyle{corollary}
\newtheorem{corollary}[theorem]{Corollary}

\theoremstyle{proposition}
\newtheorem{proposition}[theorem]{Proposition}

\theoremstyle{definition}
\newtheorem{definition}[theorem]{Definition}

\theoremstyle{remark}

\numberwithin{equation}{section}

\newcommand{\comment}[1]{}
\allowdisplaybreaks

\begin{document}

\title [MAXIMAL OPERATORS, SINGULAR INTEGRALS]{MAXIMAL OPERATORS AND SINGULAR INTEGRALS ON THE WEIGHTED LORENTZ AND MORREY SPACES}

\author{Nguyen Minh Chuong}

\address{Institute of mathematics, Vietnamese  Academy of Science and Technology,  Hanoi, Vietnam.}
\email{nmchuong@math.ac.vn}

\author{Dao Van Duong}
\address{School of Mathematics,  Mientrung University of Civil Engineering, Phu Yen, Vietnam.}
\email{daovanduong@muce.edu.vn}

\author{Kieu Huu Dung}
\address{School of Mathematics, University of Transport and Communications, Ha Noi, Vietnam.}
\email{khdung@utc2.edu.vn}
\thanks{This paper is supported by Vietnam National Foundation for Science and Technology Development (NAFOSTED)}

\keywords{Maximal function, sublinear operator, strongly singular integral, commutator,  $A_p$ weight, $A(p,1)$ weight, $A_p(\varphi)$ weight, BMO space, Lorentz spaces, Morrey spaces.}
\subjclass[2010]{Primary 42B20, 42B25; Secondary 42B99}
\begin{abstract}
In this paper, we first give some new characterizations of Muckenhoupt type weights through establishing the  boundedness of maximal operators on the weighted Lorentz and Morrey spaces. Secondly, we establish the boundedness of sublinear operators including many interesting in harmonic analysis and its commutators on the weighted Morrey spaces. Finally, as an application, the boundedness of strongly singular integral operators and commutators with symbols in BMO space are also given.
\end{abstract}

\maketitle

\section{Introduction}\label{section1}
Let $f$ be a locally integrable function on $\mathbb R^n$. The Hardy-Littlewood maximal operator of $f$ is defined	by
\begin{align}
M(f)(x)=\sup\limits_{Q}\frac{1}{|Q|}\int_Q|f(y)|dy,\;\;x\in \mathbb R^n,
\end{align}
where the supremum is taken over all cubes containing $x$. It is well known that the Hardy-Littlewood maximal operator is one of the most important operators and plays a key role in harmonic analysis since maximal operators could control crucial quantitative information concerning the given functions. It is very a powerful tool for solving crucial problems in analysis, for example, applications to differentiation theory, in the theory of singular integral operators and partial differential equations (see \cite{FS1971}, \cite{St1993}, \cite{T1986} for more details). 
\vskip 5pt
It is very important to study weighted estimates for maximal operators in harmonic analysis. B. Muckenhoupt \cite{Mu1972} first discovered the weighted norm inequality for the Hardy-Littlewood maximal operators in the real setting. More precisely, it is proved that
for $1<p<\infty$, 
\begin{align}\label{maximal function}
\int\limits_{\mathbb R^n}\left|M(f)(x)\right|^p\omega(x)dx
\leq C \int\limits_{\mathbb R^n}\left|f(x)\right|^p\omega(x)dx,
\end{align}
holds for all $f$ in the weighted Lebesgue space $L^p(\omega(x)dx)$ if and only if $\omega$ belongs to the class of Muckenhoupt weights denoted by $A_p$. 
\vskip 5pt
Later, Coifman and Fefferman \cite{CF1974} extended the theory of Muckenhoupt weights  to general Calder\'{o}n-Zygmund operators. They also proved that $A_p$ weights satisfy the crucial reverse H\"{o}lder condition. The weighted norm inequalities for the maximal operators are also extended to the vector valued setting by Andersen and John in the work \cite{AJ1981}, and to the Lorentz spaces by  Chung,  Hunt, and Kurtz in \cite{CHK1982}. It is well known that the theory of weighted functions plays an important role in the study of boundary value problems on Lipschitz domains, in theory of extrapolation of operators and applications to certain classes of nonlinear partial differential equation. 
\vskip 5pt
It is also useful to remark that in 2012, Tang \cite{Ta2012} established the weighted norm inequalities for maximal operators and pseudodifferential operators with smooth symbols associated to the class of new weighted functions $A_p(\varphi)$ (see in Section \ref{section2} below for more details) including the Muckenhoupt weighted functions.
It should be pointed out that the class of $ A_p(\varphi)$ weights do not satisfy the doubling condition. 
\vskip 5pt
It is well known that Morrey \cite{Mo1938} introduced the classical Morrey spaces to study the local behavior of solutions to second order elliptic partial differential equations. Moreover, it is found that many properties of solutions to partial differential equations can be attributed to the boundedness of some operators on Morrey spaces. Also, the Morrey spaces have many important applications to Navier-Stokes  and Schr\"{o}dinger equations, elliptic equations with discontinuous coefficients and potential theory (see, for example, \cite{Adams1975}, \cite{Caffarelli1988},  \cite {Fan1998}, \cite{Mazzucato03}, \cite{Ruiz1991}, \cite{T1992} and therein references).
During last decades, the theory of Morrey spaces has been significantly developed
into different contexts, including the study of classical operators of harmonic analysis, for instance, maximal functions, potential operators, singular  integrals, pseudodifferential operators, Hausdorff operators and their commutators in generalizations of these spaces (see \cite{AGL2000}, \cite{Ch2018}, \cite{Guliyev2011}, \cite{G2016}, \cite{KS2009}). Especially,  Wang, Zhou and Chen \cite{WZC2017} recently have established the interesting connection between the $A_p$ weights and Morrey spaces. More precisely, some new characterizations of Muckenhoupt weights are given by replacing the Lebesgue spaces by the Morrey spaces. Motivated by all of the above mentioned facts, the first main of this paper is to give some new characterizations of Muckenhoupt type weights such as $A_p$, $A(p,1)$, and $A_p(\varphi)$ by establishing the boundedness of maximal operators on the weighted Morrey and Lorentz spaces. In particular, we give the weighted norm inequality of weak type for new dyadic maximal operators associated to the $A_p^{\Delta,\eta}(\varphi)$ dyadic weights. The results are given in Section \ref{section3} of the paper.
\vskip 5pt
The second main of this paper is to study the boundedness of sublinear operators  including many interesting operators in harmonic analysis, such as the Calder\'{o}n-Zygmund operator, Hardy-Littlewood maximal operator, strongly singular integrals, and  so on, on the weighted Morrey spaces.
\vskip 5pt
Let us first give the definition of sublinear operators with strongly singular kernels. Let the operator $\mathcal{T}$ be well defined on the space of all infinitely differential functions with compact support $C^\infty_c(\mathbb R^n)$. It is said that $\mathcal{T}$ is a strongly singular sublinear operator if it is a linear or sublinear operator and satisfies the size condition as follows
\begin{align}\label{ineq-sub}
\left|\mathcal{T}f(x)\right|\leq C\int_{\mathbb R^n}\frac{|f(y)|}{|x-y|^{n+\lambda}}dy, \text{\; for\; a.e \;} x\not\in \text{supp}{f},
\end{align}
for all $f\in C^\infty_c(\mathbb R^n)$, where $\lambda$ is a non-negative real number. 
\vskip 5pt
For a measurable function $b$,  the commutator operator $[b, \mathcal{T}]$
is defined as a linear or a sublinear operator such that
\begin{align}\label{ineq-sub-com}
\left|[b, \mathcal{T}]f(x)\right|\leq C\int_{\mathbb R^n}\frac{|f(y)||b(x)-b(y)|}{|x-y|^{n+\lambda}}dy, \text{\; for\; a.e \;} x\not\in \text{supp}{f},
\end{align}
for every $f\in C^\infty_c(\mathbb R^n)$. For $\lambda\leq 0$, the sublinear operators $\mathcal{T}$ and $[b, \mathcal{T}]$ have been investigated by many authors. For example, see in the works \cite{Guliyev2011}, \cite{Kokilashvili2016}, \cite{Soria1994} and therein references. In the Section \ref{section4} of the paper, we establish the boundedness of sublinear operators $\mathcal{T}$ and $[b, \mathcal{T}]$ for $\lambda\geq 0$ on the weighted Morrey type spaces. As an application, we obtain some new results about boundedness of strongly singular integral operators and their commutators with symbols in BMO space on the weighted Morrey spaces. Moreover, maximal singular integral  operators of Andersen and John type are studied on the two weighted  Morrey spaces with vector valued functions in Section \ref{section4}.
\section{Some notations and definitions}\label{section2}
Throught the whole paper, we denote by $C$ a positive geometric constant that is independent of the main parameters, but can change from line to line.
We also write $a\lesssim b$ to mean that there is a positive constant $C$, independent of the main parameters, such that $a \le Cb$. The symbol $f\simeq g$ means that f is equivalent to g (i.e. $C^{-1} f\leq g \leq Cf)$. As usual, $\omega(\cdot)$ is a non-negative weighted function on $\mathbb{R}^n$. Denote $\omega(B)^{\alpha }=\big(\int_B\omega(x)dx\big)^{\alpha}$, for $\alpha\in\mathbb R$. Remark that if $\omega(x) = x^{\beta}$ for $\beta > -n$, then we have
\begin{align}\label{ineq-power}
\omega(B_r(0))=\int_{B_r(0)}|x|^{\beta}dx\simeq r^{\beta+n}.
\end{align}
We also denote by $B_r(x_0)=\{x\in\mathbb R^n:|x-x_0|<r\}$ a ball of radius $r$ with center at $x_0$, and let $rB$ define the ball with the same center as $B$ whose radius is $r$ times radius of $B$.

Now, we are in a position to give some notations and definitions of weighted Morrey spaces.

\begin{definition}
Let $1 \le q < \infty, 0 < \kappa < 1$ and  $\omega_1$  and $\omega_2$ be two weighted functions. Then two weighted Morrey space is defined	by
\begin{align*}
\mathcal{B}^{q,\kappa}_{\omega_1,\omega_2}(\mathbb{R}^n)=\{f\in L^q_{\omega_2,{\rm loc}}(\mathbb{R}^n):\|f\|_{\mathcal{B}^{q,\kappa}_{\omega_1,\omega_2}(\mathbb{R}^n)}<\infty \},
\end{align*}
where
\begin{align*}
\|f\|_{\mathcal{B}^{q,\kappa}_{\omega_1,\omega_2}(\mathbb{R}^n)}=\sup\limits_{\rm ball\,B} \Big(\frac{1}{\omega_1(B)^{\kappa}}\int_{B}|f(x)|^q\omega_2(x)dx \Big)^{\frac{1}{q}}.
\end{align*}
\end{definition}
 It is easy to see that $\mathcal{B}^{q,\kappa}_{\omega_1,\omega_2}(\mathbb{R}^n)$ is a Banach space. Note that if $\omega_1=\omega, \omega_2=1$,  we then write $\mathcal B^{q,\kappa}(\omega,\mathbb R^n):=\mathcal B^{q,\kappa}_{\omega_1,\omega_2}(\mathbb R^n)$. Also, if $\omega_1=\omega_2=\omega$, then we denote $\mathcal B^{q,\kappa}
_\omega(\mathbb R^n):= {\mathcal B}^{q,\kappa}_{\omega_1,\omega_2}(\mathbb R^n)$. In particular, for $\omega=1$ we write $\mathcal B^{q,\kappa}
(\mathbb R^n):=\mathcal B^{q,\kappa}_\omega(\mathbb R^n)$.
\begin{definition}
Let $1 \le q < \infty, 0 < \kappa < 1$. The local Morrey space is defined by
\begin{align*}
\mathcal{B}^{q,\kappa}_{\rm loc}(\mathbb{R}^n)=\{f\in L^q_{{\rm loc}}(\mathbb{R}^n):\|f\|_{\mathcal{B}^{q,\kappa}_{\rm loc}(\mathbb{R}^n)}<\infty \},
\end{align*}
where
\begin{align*}
\|f\|_{\mathcal{B}^{q,\kappa}_{\rm loc}(\mathbb{R}^n)}=\sup\limits_{x\in\mathbb{R}^n,0<R<1} \Big(\frac{1}{|B_R(x)|^{\kappa}}\int_{B_R(x)}|f(y)|^q dy \Big)^{\frac{1}{q}}.
\end{align*}
\end{definition}
Note that for $1\leq q\leq p<\infty$, the local Morrey space  $\mathcal{B}^{q,1-\frac{q}{p}}_{\rm loc}(\mathbb{R}^n)$ has some important applications to the Navier-Stokes equations and other evolution equations (see in \cite{ Fe1993, T1992} for more details).
\begin{definition}
Let $1 \le q < \infty$ and $0 < \kappa < 1$. The weighted  inhomogeneous Morrey space is defined by
\begin{align*}
{B}^{q,\kappa}_\omega(\mathbb{R}^n)=\{f\in L^q_{{\omega, \rm loc}}(\mathbb{R}^n):\|f\|_{{B}^{q,\kappa}_\omega(\mathbb{R}^n)}<\infty \},
\end{align*}
where
\begin{align*}
\|f\|_{{B}^{q,\kappa}_{\omega}(\mathbb{R}^n)}=\sup\limits_{x\in\mathbb{R}^n,R\geq 1} \Big(\frac{1}{\omega(B_R(x))^{\kappa}}\int_{B_R(x)}|f(x)|^q\omega(x) dx \Big)^{\frac{1}{q}}.
\end{align*}
\end{definition}
If $\omega=1$ and $1\leq q\leq p<\infty$, then the inhomogeneous Morrey space  ${B}^{q,1-\frac{q}{p}}_\omega(\mathbb{R}^n)$ is introduced by  Alvarez,  Guzm\'{a}n-Partida and  Lakey (see in \cite{AGL2000} for more details). Note that $\mathcal{B}^{q,\kappa}_{\rm loc}(\mathbb{R}^n)$ and ${B}^{q,\kappa}_\omega(\mathbb{R}^n)$ are two Banach spaces.
\begin{definition} 
Let $1 \le q < \infty, 0 < \kappa < 1$ and $\omega$ be a weighted function. The  weighted Morrey space is defined by
\begin{align*}
\mathcal{L}^{q,\kappa}_{\omega}(\mathbb{R}^n)=\{f\in L^q_{\omega,{\rm loc}}(\mathbb{R}^n):\|f\|_{\mathcal{L}^{q,\kappa}_{\omega}(\mathbb{R}^n)}<\infty \},
\end{align*}
where
\begin{align*}
\|f\|_{\mathcal{L}^{q,\kappa}_{\omega}(\mathbb{R}^n)}=\sup\limits_{\rm cube ~ Q} \Big(\frac{1}{\omega(Q)^k}\int_{Q}|f(x)|^q\omega(x) dx \Big)^{\frac{1}{q}}.
\end{align*}
\end{definition}
 
From this, for convenience, we denote $M^p_{q,\omega}(\mathbb{R}^n):=\mathcal{L}^{q,1-\frac{q}{p}}_{\omega}(\mathbb{R}^n)$ for the case $0 < q < p < \infty$.

\begin{definition}
Let $0 < q \le p < \infty$ and $\omega$ be a weighted function. Then the weighted weak Morrey space is defined	by
\begin{align*}
WM^{p}_{q,\omega}(\mathbb{R}^n)=\{f\in L^q_{\omega,{\rm loc}}(\mathbb{R}^n):\|f\|_{WM^{p}_{q,\omega}(\mathbb{R}^n)}<\infty \},
\end{align*}
where
\begin{align*}
\|f\|_{WM^{p}_{q,\omega}(\mathbb{R}^n)}=\sup\limits_{\rm cube ~ Q}\frac{1}{\omega(Q)^{\frac{1}{q}-\frac{1}{p}}}\sup\limits_{\lambda>0}\lambda \Big(\int_{\{x\in Q:|f(x)|>\lambda \}}\omega(x) dx \Big)^{\frac{1}{q}}.
\end{align*} 
\end{definition}

For a measurable function $f$ on $\mathbb{R}^n$, the distribution function of $f$ associated with the measure $\omega(x)dx$ is defined as follows
\begin{align*}
d_f(\alpha)=\omega\left(\{x\in \mathbb{R}^n: |f(x)|>\alpha \}\right).
\end{align*}

The decreasing rearrangement of $f$ with respect to the measure $\omega(x)dx$ is the function $f^*$ defined	on $[0, \infty)$ as follows
\begin{align*}
f^*(t)=\inf \{s>0:d_f(s)\le t \}.
\end{align*}

\begin{definition}
(Section 2 in \cite{CHK1982}). Let $0 < p, q \le \infty$. The weighted Lorentz space $L^{p,q}_\omega (\mathbb{R}^n)$ is defined	as the set of all measurable functions $f$ such that
$\|f\|_{L^{p,q}_\omega (\mathbb{R}^n)}<\infty$, where
\begin{align*}
\|f\|_{L^{p,q}_\omega (\mathbb{R}^n)}=\begin{cases}
\left(\frac{q}{p}\int_0^\infty \left[t^{\frac{1}{p}}f^*(t) \right]^q\frac{dt}{t} \right)^{\frac{1}{q}}, &{\rm if} ~ 0<q<\infty,\\
\sup\limits_{t>0}t^{\frac{1}{p}}f^*(t), &{\rm if}~ q=\infty.\end{cases}
\end{align*}
\end{definition} 

Remark that if either $1 < p < \infty$ and $1 \le q \le \infty$, or $p = q = 1$, or $p = q = \infty$ then $L^{p,q}_\omega (\mathbb{R}^n)$ is a quasi-Banach space.  Moreover, there is a constant $C > 0$ such that
\begin{align}\label{Lor-ineq}
C^{-1}\|f\|_{L^{p,q}_\omega (\mathbb{R}^n)}\le \sup\limits_{\|g\|_{L^{p',q'}_\omega(\mathbb R^n)}\le 1}\left|\int_{\,\mathbb{R}^n}f(x)g(x)\omega(x)dx  \right|\le C\|f\|_{L^{p,q}_\omega (\mathbb{R}^n)}.
\end{align}

\begin{corollary}\label{Cor2.3-WZC2017}
{\rm(page 253 in \cite{H1966} and Corollary 2.3 in \cite{WZC2017})} If $0 < r < q < p < \infty$, $1\leq q_1\leq q_2\leq \infty$ and $\omega$ is a non-negative weighted function on $\mathbb{R}^n$, then there exists a constant $C > 0$ such that
\begin{align*}
C\|\cdot\|_{M^p_{r,\omega}(\mathbb{R}^n)}&\le \|\cdot\|_{WM^p_{q,\omega}(\mathbb{R}^n)}\le \|\cdot\|_{M^p_{q,\omega}(\mathbb{R}^n)}\le \|\cdot\|_{WM^p_{p,\omega}(\mathbb{R}^n)}\nonumber
\\
&=\|\cdot\|_{L^{p,\infty}_\omega (\mathbb{R}^n)}\leq \|\cdot\|_{L^{p,q_2}_\omega (\mathbb{R}^n)}\leq \|\cdot\|_{L^{p,q_1}_\omega (\mathbb{R}^n)}.
\end{align*}
\end{corollary}

Next, we present some basic facts on the class of weighted functions $A(p, 1)$ with $1 < p  < \infty$.  For further information on the weights, the interested readers may refer to the work \cite{CHK1982}. The weighted function $\omega(x)$ is in $A(p, 1)$ if there exists a positive constant $C$ such that for any cube $Q$, we have
\begin{align*}
\|\chi_Q \|_{L^{p,1}_\omega(\mathbb{R}^n)} \|\chi_Q\omega^{-1} \|_{L^{p',\infty}_\omega (\mathbb{R}^n)}\le C|Q|.
\end{align*}
\begin{lemma}\label{Lem2.8-CHK1982}
{\rm(Lemma 2.8 in \cite{CHK1982})} For $1\leq p<\infty$, we have $\omega \in A(p, 1)$ if and only if there exists a constant $C$ such that for any cube $Q$ and subset $E \subset Q$,
\begin{align*}
\frac{|E|}{|Q|}\le C\left(\frac{\omega(E)}{\omega(Q)} \right)^{\frac{1}{p}}.
\end{align*}
\end{lemma}
Remark that if $\omega\in A(p,1)$ with $1\leq p<\infty$ and $0<\kappa<1$, then $\mathcal{B}^{p,\kappa}_\omega(\mathbb R^n)=\mathcal{L}^{p,\kappa}_{\omega}(\mathbb{R}^n)$ with equivalence of norms.
\\

Let $1 \le r < \infty$ and $\vec{f}=\{f_k\}$ be a sequence of measurable functions on $\mathbb{R}^n$.
We denote
\begin{align*}
|\vec{f}(x)|_r=\left(\sum\limits_{k=1}^\infty |f_k(x)|^r \right)^{\frac{1}{r}}.
\end{align*} 

As usual, the vector-valued space $X(\ell^r,\mathbb{R}^n)$ is defined as the set of all sequences of measurable functions $\vec{f}=\{f_k\}$ such that
\begin{align*}
\|\vec{f} \|_{X(\ell^r,\mathbb{R}^n)}=\||\vec{f}(\cdot)|_r \|_X<\infty,
\end{align*}
where $X$ is an appropriate Banach space.

Let us recall to define the BMO spaces of John and Nirenberg. For further
information on these spaces as well as their deep applications in harmonic analysis,
 one can see in the famous book of Stein \cite{St1993}.

\begin{definition}
The bounded mean oscillation space $BMO(\mathbb{R}^n)$ is defined as the set of all functions $b\in L^1_{\rm loc}(\mathbb{R}^n)$ such that
\begin{align*}
\|b\|_{BMO(\mathbb{R}^n)}=\sup\limits_{\rm cube ~Q}\frac{1}{|Q|}\int_Q |b(x)-b_Q|dx<\infty,
\end{align*}
where $b_Q=\frac{1}{|Q|}\int_Q b(x)dx$.
\end{definition}
\begin{lemma}\label{BMO-Lemma}
{\rm (\cite{St1993})}
If $1<p<\infty$,  we then have
\begin{align*}
\|b\|_{BMO(\mathbb{R}^n)}\simeq \sup\limits_{\rm cube ~Q}\Big(\frac{1}{|Q|}\int_Q |b(x)-b_Q|^pdx\Big)^{\frac{1}{p}}:=\|b\|_{BMO^p(\mathbb R^n)}.
\end{align*}
\end{lemma}
\begin{proposition}\label{Pro-T1986}
{\rm (Proposition 3.2 in \cite{T1986})}
If $b\in BMO(\mathbb R^n)$, then 
$$
|b_{2^{j+1}B}-b_B|\leq 2^n(j+1)\|b\|_{BMO(\mathbb R^n)},\,\textit{\rm for all}\, j\in\mathbb N.
$$
\end{proposition}
Let us recall the definition of $A_p$ weights. For further readings on $A_p$ weights, the reader may find in the interesting book \cite{Grafakos2008}.
\begin{definition}
Let $1 < p < \infty$. It is said that a weight $\omega \in A_p(\mathbb{R}^n)$ if there exists a constant $C$ such that for all cubes $Q$,
\begin{align*}
\left(\frac{1}{|Q|}\int_Q \omega(x)dx \right) \left(\frac{1}{|Q|}\int_Q\omega(x)^{-\frac{1}{p-1}}dx \right)^{p-1}\le C.
\end{align*}
\end{definition}
A weight $\omega \in  A_1(\mathbb{R}^n)$ if there is a constant $C$ such that
\begin{align*}
M(\omega)(x)\le C\omega(x), \;{ \rm for \; a. e} ~x\in \mathbb{R}^n.
\end{align*}
We denote $A_\infty(\mathbb{R}^n)=\mathop\cup\limits_{1\le p<\infty}A_p(\mathbb{R}^n)$.
 \\
 
A closing relation to $A_\infty(\mathbb{R}^n)$ is the reverse H\"{o}lder condition. If there exist $r > 1$ and a fixed	constant $C$ such that
\begin{align*}
\Big(\frac{1}{|B|}\int_B \omega(x)^rdx \Big)^{\frac{1}{r}}\le \frac{C}{|B|}\int_B \omega(x) dx,
\end{align*}
for all balls $B \subset \mathbb{R}^n$, we then say that $\omega$ satisfies the reverse H\"{o}lder condition of order $r$ and write $\omega\in  RH_r (\mathbb{R}^n)$. According to Theorem 19 and Corollary 21 in \cite{IMS2015}, $\omega\in A_\infty (\mathbb{R}^n)$ if and only if there exists some $r  > 1$ such that
$\omega\in  RH_r (\mathbb{R}^n)$.  Moreover, if $\omega\in  RH_r (\mathbb{R}^n),r>1$, then $\omega\in  RH_{r+\varepsilon} (\mathbb{R}^n)$ for some $\varepsilon>0$. We thus write $r_\omega  = \sup\{r > 1 : \omega\in  RH_r (\mathbb{R}^n)\}$ to denote the
critical index of $\omega$ for the reverse H\"{o}lder condition.
\begin{proposition}\label{rever-Holder}
Let $\omega \in A_p(\mathbb{R}^n) \cap RH_r(\mathbb{R}^n), p \ge 1$ and $r > 1$. Then, there exist two constants $C_1, C_2 > 0$ such that
\begin{align*}
C_1\left( \frac{|E|}{|B|}\right)^p\le \frac{\omega(E)}{\omega(B)} \le C_2\left( \frac{|E|}{|B|}\right)^{\frac{r-1}{r}},
\end{align*}
for any ball $B$ and for any measurable subset $E$ of $B$.
\end{proposition} 
\begin{proposition}\label{pro2.4DFan}
If $\omega\in A_p(\mathbb R^n)$, $1 \leq p < \infty$, then for any $f\in L^1_{\rm loc}(\mathbb R^n)$ and any ball $B \subset \mathbb R^n$, we have
$$
\dfrac{1}{|B|}\int_{B}|f(x)|dx\leq C\Big(\dfrac{1}{\omega(B)}\int_{B}|f(x)|^p\omega(x)dx \Big)^{\frac{1}{p}}.
$$
\end{proposition}
Next, we write $\omega\in \Delta_2$, the class of doubling weights, if there exists $D > 0$ such that for any cube $Q$, we have
\begin{align*}
\omega(2Q)\le D\omega(Q).
\end{align*}
It is known that if $\omega\in A_\infty(\mathbb{R}^n)$ then $\omega\in \Delta$. 
Now, let us recall the class of $ A_p(\varphi)$ weights proposed by Tang in the work \cite{Ta2012}.

\begin{definition}
Let $1 < p < \infty$ and $\varphi(t)=(1+t)^{\alpha_0}$  for $\alpha_0>0$ and $t\ge 0$. We say that a weight $\omega\in A_p(\varphi)$ if there exists a constant $C$ such that for all cubes $Q$,
\begin{align*}
\Big(\frac{1}{\varphi(|Q|)|Q|}\int_Q \omega(x)dx \Big).\Big(\frac{1}{\varphi(|Q|)|Q|}\int_Q\omega(x)^{-\frac{1}{p-1}}dx \Big)^{p-1}\le C.
\end{align*}
A weight $\omega\in A_1(\varphi)$ if there is a constant $C$ such that
$$
M_\varphi(f)(x)\le C\omega(x), {\rm\; for\; a. e} ~x\in \mathbb{R}^n,
$$
where
\begin{align*}
M_\varphi(f)(x)=\sup\limits_{x\in {\rm cube }~Q}\frac{1}{\varphi(|Q|)|Q|}\int_Q |f(y)|dy.
\end{align*} 
\end{definition}
Denote $A_\infty(\varphi)=\mathop\cup\limits_{1\le p<\infty} A_p(\varphi)$. It is useful to remark that the  $ A_p(\varphi)$ weights do not satisfy the doubling condition. For instance, $\omega(x)=(1+|x|)^{(-n+\eta)}$ for $0\leq\eta\leq n\alpha_0$ is in $A_1(\varphi)$, but not in $A_p$ weights and $\omega(x)dx$ is not a doubling measure. It is also important to see that $M_\varphi$ may be not bounded on the weighted Lebesgue spaces $L^p_\omega(\mathbb R^n)$ for every $\omega\in A_p(\varphi)$. To be more precise, see in Lemma 2.3 in \cite{Ta2012}. Similarly, in this paper we also introduce a class of dyadic weighted functions associated to the function $\varphi$ as follows.
\begin{definition} 
Let $1<p<\infty$ and $0<\eta<\infty$. A weight $\omega\in A_p^{\Delta,\eta}(\varphi)$ if there exists a constant $C$ such that for all dyadic cubes $Q$,
\begin{align*}
\Big(\frac{1}{\varphi(|Q|)^{\eta}|Q|}\int_Q \omega(x)dx \Big).\Big(\frac{1}{\varphi(|Q|)^{\eta}|Q|}\int_Q\omega(x)^{-\frac{1}{p-1}}dx \Big)^{p-1}\le C.
\end{align*}
\end{definition}

It is obvious that  $A_{p_1}^{\Delta,\eta}(\varphi)\subset A_{p_2}^{\Delta,\eta}(\varphi)$ for all $1<p_1<p_2<\infty$. It is also easy to show that $A_p(\mathbb{R}^n)\subset A_p(\varphi)\subset A_p^{\Delta,\eta}(\varphi)$ with $1<p<\infty$ and $0< \eta<\infty$. In particular,  $A_1(\mathbb R^n) \subset A_1(\varphi)$. 

Next, we give the definitions of the maximal operators $M_\omega$ and $M^\Delta_{\varphi,\eta}$ as follows
$$
M_\omega(f)(x)=\sup\limits_{x\in\text{ball\;} B}\frac{1}{\omega(5B)}\int_B|f(y)|\omega(y)dy,$$
$$M^\Delta_{\varphi,\eta}(f)(x)=\sup\limits_{x\in {\rm dyadic\,cube}\,Q}\frac{1}{\varphi(|Q|)^\eta |Q|}\int_Q |f(y)|dy, {\;\rm for~all}~0<\eta<\infty.
$$
Remark that by the similar arguments to Lemma 2.1 in \cite{Ta2012}, we also have
$$M^\Delta_{\varphi,\eta}(f)(x)\lesssim \left(M_\omega(|f|^p)(x)\right)^{\frac{1}{p}}, \;\;x\in\mathbb R^n,$$
where $\omega\in A_{p}^{\Delta,\eta}(\varphi)$ for $0<\eta<\infty$ and $1<p<\infty$. Moreover, we also get the same result as in Lemma 2.3 of the paper \cite{Ta2012}. 
\begin{lemma}\label{Lemma-dyadic}
Let $1<p<\infty$ and $\omega\in A_p^{\Delta,\eta}(\varphi)$. Then, for any $p<r<\infty$, we have
$$
\|M^\Delta_{\varphi,\eta}(f)\|_{L^{r}_\omega(\mathbb R^n)}\leq C \|f\|_{L^{r}_\omega(\mathbb R^n)}.
$$
\end{lemma}
It seems to see that the inequality in Lemma \ref{Lemma-dyadic} may be not valid for $r=p$.
\begin{theorem}\label{Theo3.1-AJ1981}
{\rm (Theorem 3.1 in \cite{AJ1981})} If $1 < p < \infty$, then the operator $M$  is bounded from $L^p_\omega(\ell^r, \mathbb{R}^n)$ to itself if and only if $\omega\in A_p$.
\end{theorem}
\begin{theorem}\label{Theo2.12-CDD2017}
{\rm(Theorem 2.12 in \cite{CDD2017})} If $1 < p < r < \infty$, then the operator $M$ is
bounded from $L^{p,1}_\omega(\ell^r, \mathbb{R}^n)$ to $L^{p,\infty}_\omega(\ell^r, \mathbb{R}^n)$ if and only if $\omega\in A(p, 1)$.
\end{theorem}
\begin{lemma}\label{Lem2.3-Ta2012}
{\rm(Lemma 2.3 in \cite{Ta2012})} If $1 \le p < \infty$, then the operator $M_\varphi$ is bounded from $L^p_\omega(\mathbb{R}^n)$ to $L^{p,\infty}_\omega(\mathbb{R}^n)$ if and only if $\omega\in A_p(\varphi)$.
\end{lemma}
In 1981, Andersen and John \cite{AJ1981} established the weighted norm inequalities for vector-valued maximal functions and maximal singular integrals on the space $L^p_\omega(\ell^r,\mathbb R^n)$. Now, let us recall the definition of maximal singular integrals associated to the kernels due to Andersen and John. For more details, see in the work \cite{AJ1981}.

\begin{definition}
Let $K$ be the kernel such that
\begin{align}
|K(x)|&\le \frac{A}{|x|^n}, |\hat{K}(x)|\le A;\\
|K(x-y)-K(x)|&\le \mu (|y|/|x|)|x|^{-n}, \text{for all} ~|x|\ge 2|y|;
\end{align}
where $A$ is a constant and $\mu$ is non-decreasing on the positive real half-line,
$\mu(2t)\le C\mu (t)$ for all $t > 0$, and satisfies the Dini condition
\begin{align}
\int_0^1\frac{\mu(t)}{t}dt<\infty.
\end{align}
Then, the maximal singular integral operator $T^*$ is defined	by
\begin{align*}
T^*(f)(x)= \mathop{\rm sup}\limits_{\varepsilon >0}\Big|\int_{\,\,|x-y|\geq \varepsilon}K(x-y)f(y)dy\Big|.
\end{align*}
\end{definition}
 
If $\{K_k (x)\}$ denote a sequence of singular convolution kernels satisfying the above conditions (2.3)-(2.5) with a uniform constant $A$ and a fixed function $\mu$ not dependent of $k$, then we write $T^*(\vec{f})=\{T^*_k(f_k) \}$, where $T^*_k$  is the operator above corresponding to the kernel $K_k$ .
\begin{theorem}\label{Theo-AJ1981}
{\rm(Theorem 5.2 in \cite{AJ1981})} Let $1 < r < \infty, 1 < p < \infty,$ and suppose $\omega\in A_p$. There exits a constant $C$ such that
\begin{align*}
\|T^*(\vec{f}) \|_{L^p_\omega(\ell^r,\mathbb{R}^n)}\le C\|\vec{f} \|_{L^p_\omega(\ell^r,\mathbb{R}^n)}, \text{ for all} ~f\in L^p_\omega(\ell^r,\mathbb{R}^n). 
\end{align*}
\end{theorem}
Let $b$ be a measurable function. We denote by $\mathcal{M}_b$ the multiplication operator defined  by $\mathcal{M}_bf (x)=b(x) f (x)$ for any measurable function $f$. If $\mathcal{H}$ is a linear or sublinear operator on some measurable function space, the commutator of Coifman-Rochberg-Weiss type formed by $\mathcal{M}_b$  and $\mathcal{H}$ is defined by $[\mathcal{M}_b, \mathcal{H}]f (x)=(\mathcal{M}_b\mathcal{H}-\mathcal{H}\mathcal{M}_b) f (x)$. 
\section{The results about the boundedness of maximal operators}\label{section3}
By using Theorem \ref{Theo3.1-AJ1981} and estimating as Theorem 1.1 in \cite{WZC2017}, we immediately have the following useful characterization for the Muckenhoupt weights through boundedness of the Hardy-Littlewood maximal operators on the vector valued function spaces.

\begin{theorem}
Let $1 < q < p < \infty, 1 \le r < \infty$. Then, the following statements are equivalent:
\begin{enumerate}
\item[(1)]	$\omega\in A_p$;
\item[(2)]	$M$  is a bounded operator from $L^p_\omega(\ell^r,\mathbb{R}^n)$ to $L^{p,\infty}_\omega(\ell^r,\mathbb{R}^n)$;
\item[(3)]	$M$  is a bounded operator from $L^p_\omega(\ell^r,\mathbb{R}^n)$ to $M^{p}_{q,\omega}(\ell^r,\mathbb{R}^n)$;
\item[(4)]	$M$  is a bounded operator from $L^p_\omega(\ell^r,\mathbb{R}^n)$ to $WM^{p}_{q,\omega}(\ell^r,\mathbb{R}^n)$.
\end{enumerate}
\end{theorem}
Now, we give a new characterization for the class of $A(p,1)$ weights.
\begin{theorem}
Let $1 < q  < p  < r  < \infty$.	The following statements are equivalent:
\begin{enumerate}
\item[(1)]	$\omega\in A(p,1)$;
\item[(2)]	$M$  is a bounded operator from $L^{p,1}_\omega(\ell^r,\mathbb{R}^n)$ to $L^{p,\infty}_\omega(\ell^r,\mathbb{R}^n)$;
\item[(3)]	$M$  is a bounded operator from $L^{p,1}_\omega(\ell^r,\mathbb{R}^n)$ to $M^{p}_{q,\omega}(\ell^r,\mathbb{R}^n)$;
\item[(4)]	$M$  is a bounded operator from $L^{p,1}_\omega(\ell^r,\mathbb{R}^n)$ to $WM^{p}_{q,\omega}(\ell^r,\mathbb{R}^n)$.
\end{enumerate}
\end{theorem}
\begin{proof}
Note that Theorem \ref{Theo2.12-CDD2017} follows us to obtain the equivalence of (1) and (2). By Corollary \ref{Cor2.3-WZC2017}, we immediately have (2) $\Rightarrow$ (3) $\Rightarrow$ (4). Therefore, to complete the proof of the theorem, we need to prove (4) $\Rightarrow$ (1). For any cube $Q$, by the relation (\ref{Lor-ineq}), we find a function $f$ such that $\|f\|_{L^{p,1}_\omega(\mathbb{R}^n)}\le 1$ and
\begin{align}\label{ineq-Ap1}
\int_Q|f(x)|dx\ge \left|\,\int_{\mathbb{R}^n}f(x)\chi_Q\omega^{-1}\omega dx  \right|\gtrsim \|\chi_Q\omega^{-1} \|_{L^{p',\infty}_\omega(\mathbb{R}^n)}.
\end{align}
It is obvious that $Q=\{x\in Q: M(f)(x)>\lambda \}$, where $\lambda=\frac{1}{2|Q|}\int_Q|f(x)|dx$. Thus, because $M$ is a bounded operator from $L^{p,1}_\omega(\mathbb{R}^n)$ to $WM^{p}_{q,\omega}(\mathbb{R}^n)$, we have
\begin{align*}
\lambda\omega(Q)^{\frac{1}{p}}&=\frac{1}{\omega(Q)^{\frac{1}{q}-\frac{1}{p}}}\lambda\Big(\int_{\{x\in Q: M(f)(x)>\lambda \}}\omega(x)dx \Big)^{\frac{1}{q}}\\
&\le \|M(f) \|_{WM^p_{q,\omega}(\mathbb{R}^n)}\le \|f\|_{L^{p,1}_\omega(\mathbb{R}^n)}\le 1.
\end{align*}
As a consequence, by (\ref{ineq-Ap1}), we give
\begin{align*}
\frac{1}{2|Q|}\|\chi_Q\omega^{-1} \|_{L^{p',\infty}_\omega(\mathbb{R}^n)}.\omega(Q)^{\frac{1}{p}}\lesssim 1.
\end{align*}
From this, we have
\begin{align*}
\|\chi_Q\|_{L^{p,1}_\omega(\mathbb{R}^n)}\|\chi_Q\omega^{-1} \|_{L^{p',\infty}_\omega(\mathbb{R}^n)}\lesssim |Q|.
\end{align*}
This implies that $\omega\in A(p,1)$, and  the theorem is completely proved.
\end{proof}
Next, we establish the boundedness results for pseudo-differential operators of order $0$ on weighted Lorentz spaces.
For $m\in\mathbb R$, we say that the function $a(x,\xi)\in C^{\infty}(\mathbb R^n\times \mathbb R^n)$ is a
symbol of order $m$ if it satisfies the following inequality
$$
|\partial^{\beta}_x\partial^{\alpha}_\xi a(x,\xi)|\leq C_{\alpha,\beta}(1+|\xi|)^{m-|\alpha|},
$$
for all multi-indices $\alpha$ and $\beta$, where $C_{\alpha,\beta} > 0$ is independent of $x$ and $\xi$. Then, a pseudo-differential operator is a mapping $f\to T_{a}(f)$ given by
\begin{align*}
T_a(f)(x)=\int_{\mathbb R^n} a(x,\xi)\widehat f(\xi)e^{2\pi i x\xi}d\xi.
\end{align*}
Remark that $T_a$ is well defined on the space of Schwartz functions $S(\mathbb R^n)$ or on the space of all infinitely differentiable functions with compact support $C^{\infty}_c(\mathbb R^n)$, where $\widehat f$ is the Fourier transform of the function $f$.
\begin{lemma}\label{pseudo1}
Let $1 < q \leq p < \infty$, $\omega\in A_p(\mathbb R^n)$ and $T_a$
be a  pseudo-differential operator of order $0$. Then, $T_a$ extends to a bounded operator from $L^{p,q}_\omega(\mathbb R^n)$ to  $L^{p,\infty}_\omega(\mathbb R^n)$.
\end{lemma}
\begin{proof}
By Theorem 2 and Theorem 4 in \cite{CHK1982}, we have
\begin{align}
\|M(f)\|_{L^{p,\infty}_\omega(\mathbb R^n)}\lesssim \|f\|_{L^{p,q}_\omega(\mathbb R^n)},\,\textit{\rm for all}\,f\in C^{\infty}_c(\mathbb R^n).\nonumber
\end{align}
Next, by estimating as Theorem 2 in \cite{G2016}, we see that
$$
|T_a(f)(\cdot)|\lesssim M(f)(\cdot),\,\textit{\rm for all}\,f\in C^{\infty}_c(\mathbb R^n).
$$
Thus,
\begin{align}
\|T_a(f)\|_{L^{p,\infty}_\omega(\mathbb R^n)}\lesssim \|f\|_{L^{p,q}_\omega(\mathbb R^n)},\,\textit{\rm for all}\,f\in C^{\infty}_c(\mathbb R^n).\nonumber
\end{align}
As mentioned above, since $C^{\infty}_c(\mathbb R^n)$ is dense in $L^{p,q}_\omega(\mathbb R^n)$ (see Corollary 3.2 in \cite{NTY2004}), we immediately have the desired result.
\end{proof}
By using Lemma \ref{pseudo1} and Corollary \ref{Cor2.3-WZC2017} and applying the Lorentz version Marcinkiewicz  interpolation theorem as the proof of Theorem 3 in \cite{CHK1982}, we obtain the following useful result.
\begin{theorem}\label{pseudo2}
Let $1 < p < \infty$, $1<q\leq \infty$, $\omega\in A_p(\mathbb R^n)$ and $T_a$
be a  pseudo-differential operator of order $0$. Then, the following statements are true:
\begin{enumerate}
\item[(1)] $T_a$ extends to  a bounded operator from $L^{p,q}_\omega(\mathbb R^n)$ to $L^{p,q}_\omega(\mathbb R^n)$;
\item[(2)] $T_a$  extends to  a bounded operator from $L^{p,q}_\omega(\mathbb{R}^n)$ to $M^{p}_{q,\omega}(\mathbb{R}^n)$;
\item[(3)] $T_a$  extends to  a bounded operator from $L^{p,q}_\omega(\mathbb{R}^n)$ to $WM^{p}_{q,\omega}(\mathbb{R}^n)$.
\end{enumerate}
\end{theorem}
For $1 < p < \infty$, by Lemma \ref{Lem2.8-CHK1982}, we observe that  $\omega\in A(p,1)$ implies $\omega\in \Delta$. Thus, combining with Theorem 3.1 in \cite{KS2009}, we can get the following result.
\begin{theorem}
If $1 < p < \infty, 0 < \kappa < 1, \omega\in  A(p,1)$, then the operator $M_\omega$  is bounded on $\mathcal{L}^{q,\kappa}_\omega(\mathbb{R}^n)$.
\end{theorem}
Similarly to the known characterizations of the $A_p$ weights given in \cite{WZC2017}, we also have another characterizations for the $A_p(\varphi)$ weights as follows.
\begin{theorem}\label{max-phi}
Let either $1 < q < p < \infty$ or $0 < q < p = 1$. Then, the following statements are equivalent:
\begin{enumerate}
\item[(1)]	$\omega\in A_p(\varphi)$;
\item[(2)]	$M_\varphi$  is a bounded operator from $L^{p}_\omega(\mathbb{R}^n)$ to $L^{p,\infty}_\omega(\mathbb{R}^n)$;
\item[(3)]	$M_\varphi$  is a bounded operator from $L^{p}_\omega(\mathbb{R}^n)$ to $M^{p}_{q,\omega}(\mathbb{R}^n)$;
\item[(4)]	$M_\varphi$  is a bounded operator from $L^{p}_\omega(\mathbb{R}^n)$ to $WM^{p}_{q,\omega}(\mathbb{R}^n)$.
\end{enumerate}
\end{theorem}
\begin{proof}
By Lemma \ref{Lem2.3-Ta2012} and Corollary \ref{Cor2.3-WZC2017}, it is clear that the relation (1) $\Leftrightarrow$ (2) and (2) $\Rightarrow$ (3) $\Rightarrow$ (4). Thus, to complete the proof, we need to prove the (4) $\Rightarrow$ (1). More precisely, it is as the following.
\vskip 5pt
In the case $1 < q < p < \infty$, let $Q$ be any cube and take $f_\varepsilon=(\omega+\varepsilon)^{1-p'}\chi_Q$, for all $\varepsilon>0$, where $p'$ is a 
conjugate real number of $p$, i.e $\frac{1}{p}+\frac{1}{p'}=1$. It immediately follows that $f_\varepsilon\in L^p_\omega(\mathbb{R}^n)$.
For any $0 < \lambda < \frac{(\omega+\varepsilon)^{1-p'}(Q)}{\varphi(|Q|)|Q|}$, by letting $x\in Q$, it is clear to see that
\begin{align*}
M_\varphi(f_\varepsilon)(x)\ge \frac{1}{\varphi(|Q|)|Q|}\int_Q|f_\varepsilon(y)|dy=\frac{1}{\varphi(|Q|)|Q|}\int_Q (\omega+\varepsilon)^{1-p'}dy>\lambda.
\end{align*} 
Hence, we obtain
\begin{align*}
Q=\{x\in Q: M_\varphi(f_\varepsilon)(x)>\lambda \}.
\end{align*}
Consequently, because $M_\varphi$ is a bounded operator from $L^p_\omega(\mathbb{R}^n)$ to $WM^p_{q,\omega}(\mathbb{R}^n)$, we infer
\begin{align}\label{omegaQ-weak}
\lambda\omega(Q)^{\frac{1}{p}}&=\frac{1}{\omega(Q)^{\frac{1}{q}-\frac{1}{p}}}\lambda\Big(\int_{\{x\in Q: M_\varphi(f_\varepsilon)(x)>\lambda \}}\omega(x)dx \Big)^{\frac{1}{q}}\notag\\
&\le \|M_\varphi(f_\varepsilon) \|_{WM^p_{q,\omega}(\mathbb{R}^n)}\lesssim \|f_\varepsilon\|_{L^{p}_\omega(\mathbb{R}^n)}=\Big(\int_Q (\omega+\varepsilon)^{-p'}\omega(x)dx \Big)^{\frac{1}{p}}.
\end{align}
Thus, by choosing $\lambda=\frac{(\omega+\varepsilon)^{1-p'}(Q)}{2\varphi(|Q|)|Q|}$, we get
\begin{align*}
\Big(\frac{1}{\varphi(|Q|)|Q|} \int_Q (\omega+\varepsilon)^{-p'}\omega(x)dx\Big)^p.\Big(\int_Q\omega(x)dx \Big)\lesssim \int_Q (\omega+\varepsilon)^{-p'}\omega(x)dx,
\end{align*}
which implies that
\begin{align*}
\Big(\frac{1}{\varphi(|Q|)|Q|}  \int_Q\omega(x)dx\Big) \Big( \frac{1}{\varphi(|Q|)|Q|} \int_Q (\omega+\varepsilon)^{-p'}\omega(x)dx\Big)^{p-1}\lesssim 1, 
\end{align*}
for all $\varepsilon>0$. By letting $\varepsilon\to 0^+$ and using dominated convergence theorem of Lebesgue, we obtain $\omega\in A_p(\varphi)$.
\vskip 5pt
In the case $0 < q < p = 1$, let us fix	$Q$ and take any cube $Q_1 \subset Q$. Thus, we choose $f = \chi_{Q_1}$. For any $0 < \lambda <\frac{|Q_1|}{\varphi(|Q|)|Q|}$, by estimating as (\ref{omegaQ-weak}) above, we immediately have
\begin{align*}
\lambda\Big( \int_{Q}\omega(x)dx \Big)\le \int_{Q_1}\omega(x)dx.
\end{align*}
Next, by choosing $\lambda=\frac{|Q_1|}{2\varphi(|Q|)|Q|}$, we infer
\begin{align*}
\frac{1}{|Q|}\int_Q\omega(x)dx\lesssim \frac{1}{|Q_1|}\int_{Q_1}\omega(x)dx, \text{ for any} ~Q_1\subset Q.
\end{align*}
Hence, by the definition of operator $M_\varphi$ and the Lebesgue differentiation theorem, it follows that
\begin{align*}
M_\varphi(\omega)(x)\lesssim \omega(x), \text{ for \;a.e.}\; x\in\mathbb{R}^n,
\end{align*}
which gives $\omega\in A_1(\varphi)$.
\end{proof}
In final part of this section, we give the weighted norm inequality of weak type for new dyadic maximal operators $M^\Delta_{\varphi,2\eta}$ on the  vector valued  Lebesgue spaces with weighted functions in $A_p^{\Delta,\eta}(\varphi)$.
\begin{theorem}\label{max-Delta-2eta}
If $1 < p < r < \infty, \omega \in A_p^{\Delta,\eta}(\varphi)$ for $\eta>0$, then operator $M^\Delta_{\varphi,2\eta}$ is bounded from $L^p_\omega(\ell^r,\mathbb{R}^n)$ to $L^{p,\infty}_\omega(\ell^r,\mathbb{R}^n)$.
\end{theorem}
\begin{proof}
Let $\vec{f}\in \vec S$ and $\alpha > 0$, where $\vec S$ the linear space of sequences $\vec f = \{f_k\}$ such that each $f_k(x)$ is a simple function on $\mathbb R^n$ and $f_k(x)\equiv 0$ for all sufficiently large $k$. By using Lemma 2.5 in \cite{Ta2012}, there exists a disjoint union of maximal dyadic cubes $\{Q_j\}$ such that
\begin{align}\label{ineq-fr}
|\vec{f}(x)|_r \le \alpha,x\notin \Omega=\cup_{j=1}^\infty Q_j;
\end{align}
\begin{align}\label{ineq-fr-alpha}
\alpha \le \frac{1}{\varphi(|Q_j|)^\eta|Q_j|}\int_{Q_j}|\vec{f}(x)|_rdx &\le 2^n\varphi(4n).\alpha, \text{for all} ~j\in\mathbb{Z}^+.
\end{align}
Now, we compose $\vec{f}=\vec{f'}+\vec{f^{''}}$, where $\vec{f'}=\{f'_k\},f'_k(x)=f_k(x)\chi_{\mathbb{R}^n\backslash \Omega}(x)$. This gives
\begin{align*}
|M^\Delta_{\varphi,2\eta}(\vec{f})(x)|_r\le |M^\Delta_{\varphi,2\eta}(\vec{f'})(x)|_r +|M^\Delta_{\varphi,2\eta}(\vec{f^{''}})(x)|_r.
\end{align*}
As a consequence, we need to prove the following two results
\begin{align}\label{esti-omega1}
\omega\left(\{x\in\mathbb{R}^n: |M^\Delta_{\varphi,2\eta}(\vec{f'})(x)|_r>\alpha \}  \right)\lesssim \alpha^{-p}\|\vec{f}\|^p_{L^p_\omega(\ell^r, \mathbb R^n)},
\end{align}
and
\begin{align}\label{esti-omega2}
\omega\left(\{x\in\mathbb{R}^n: |M^\Delta_{\varphi,2\eta}(\vec{f^{''}})(x)|_r>\alpha \}  \right)\lesssim \alpha^{-p}\|\vec{f}\|^p_{L^p_\omega(\ell^r, \mathbb R^n)}.
\end{align}
By Lemma \ref{Lemma-dyadic}, for $\omega \in A_p^{\Delta,\eta}(\varphi)$ we have
$$\int_{\mathbb{R}^n}|M^\Delta_{\varphi,2\eta}(f'_k)(x)|^r\omega(x)dx\leq \int_{\mathbb{R}^n}|M^\Delta_{\varphi,\eta}(f'_k)(x)|^r\omega(x)dx
\lesssim \int_{\mathbb{R}^n}|f'_k(x)|^r\omega(x)dx.$$
This implies that
\begin{align*}
\int_{\mathbb{R}^n}|M^\Delta_{\varphi,2\eta}(\vec f')(x)|^r_r\omega(x)dx&=\int_{\mathbb{R}^n}\sum_{k=1}^\infty |M^\Delta_{\varphi,2\eta}(f'_k)(x)|^r\omega(x)dx\\
&=\sum_{k=1}^\infty \int_{\mathbb{R}^n}|M^\Delta_{\varphi,2\eta}(f'_k)(x)|^r\omega(x)dx\\
&\lesssim\sum_{k=1}^\infty \int_{\mathbb{R}^n}|f'_k(x)|^r\omega(x)dx\\
&\lesssim\int_{\mathbb{R}^n}|\vec f'(x)|_r^r\omega(x)dx.
\end{align*}
Hence, by the Chebysev inequality, it immediately follows that
\begin{align}\label{ineq-Chebysev}
\omega\left(\{x\in\mathbb{R}^n: |M^\Delta_{\varphi,2\eta}(\vec{f'})(x)|_r>\alpha \}  \right)\lesssim \alpha^{-r}\|\vec{f'}\|^r_{L^r_\omega(\ell^r,\mathbb R^n)}.
\end{align}
On the other hand, by (\ref{ineq-fr}), we infer
\begin{align*}
|\vec{f'}(x)|_r^r\le \alpha^{r-p}|\vec{f}(x)|_r^p,
\end{align*}
which implies that, by (\ref{ineq-Chebysev}), the inequality (\ref{esti-omega1}) is holded.

It remains only to show that the inequality (\ref{esti-omega2}) is true. To  estimate the inequality (\ref{esti-omega2}), we put $\overline{f}=\{\overline{f}_k \}$ as follows
\begin{align*}
\overline{f}_k(x)=\begin{cases}\frac{1}{\varphi(|Q_j|)^\eta|Q_j|}\int_{Q_j}|f_k(y)|dy,& x\in Q_j, j=1,2,...,\\
0,& \text{otherwise}. \end{cases}
\end{align*} 
Then, we obtain the important inequality as follows
\begin{align}\label{Fefferman-Stein}
M^\Delta_{\varphi,2\eta}(f^{''}_k)(x)\le M^\Delta_{\varphi,\eta}(\overline{f}_k)(x),x\notin \Omega.
\end{align}
Indeed, let $x\notin \Omega$ and $Q$ be any dyadic cube such that $x \in Q$. Thus, one has
\begin{align*}
\int_Q|f^{''}_k(y)|dy=\int_{Q\cap \Omega}|f_k(y)|dy=\sum\limits_{j\in J}\int_{Q\cap Q_j}|f_k(y)|dy,
\end{align*} 
where $J = \{j \in \mathbb{N} : Q_j \cap Q \ne \emptyset\}$.  Since $\{Q_j\}$ and $Q$ are dyadic cubes, and $x \in Q$, we  immediately have $J = \{j \in \mathbb{N} : Q_j \subset Q\}$.
Hence, we infer
\begin{align}\label{int_f''}
\int_Q|f^{''}_k(y)|dy=\sum\limits_{j\in J}\int_{Q_j}|f_k(y)|dy.
\end{align}
On the other hand, we get
\begin{align*}
\int_{Q_j}\overline{f}_k(y)dy=\int_{Q_j}\Big(\frac{1}{\varphi(|Q_j|)^\eta|Q_j|} \int_{Q_j}|f_k(t)|dt\Big)dy=\frac{1}{\varphi(|Q_j|)^\eta}\int_{Q_j}|f_k(t)|dt.
\end{align*}
Therefore, by (\ref{int_f''}), one has
\begin{align}
\frac{1}{\varphi(|Q|)^{2\eta}|Q|}\int_Q|f^{''}_k(y)|dy&=\frac{1}{\varphi(|Q|)^{2\eta}|Q|}\sum\limits_{j\in J}\Big(\varphi(|Q_j|)^\eta\int_{Q_j}\overline{f}_k(y)dy \Big)\nonumber
\\
&=\frac{1}{\varphi(|Q|)^{\eta}|Q|}\sum\limits_{j\in J}\Big(\frac{\varphi(|Q_j|)^\eta}{\varphi(|Q|)^\eta}\int_{Q_j}\overline{f}_k(y)dy \Big)\nonumber
\\
&\le \frac{1}{\varphi(|Q|)^{\eta}|Q|}\int_Q\overline{f}_k(y)dy.\nonumber
\end{align}
This implies that inequality (\ref{Fefferman-Stein}) is true.
\\

Next, for any $x \in \Omega$, there only exists  a dyadic cube $Q_j$  such that $x \in Q_j$. Thus, by the Minkowski inequality and (\ref{ineq-fr-alpha}), we have
\begin{align*}
|\overline{f}(x)|_r&=\Big(\sum\limits_{k=1}^\infty\Big(\frac{1}{\varphi(|Q_j|)^{\eta}|Q_j|}\int_{Q_j}|f_k(y)|dy  \Big)^r \Big)^{\frac{1}{r}}\\
&\le \frac{1}{\varphi(|Q_j|)^{\eta}|Q_j|}\int_{Q_j}|\vec{f}(y)|_rdy\le 2^n\varphi(4n).\alpha.
\end{align*} 
Hence, by using (\ref{Fefferman-Stein}) and estimating as (\ref{ineq-Chebysev}), it is clear to see that
\begin{align*}
&\omega\big(\{x\notin\Omega: |M^\Delta_{\varphi,2\eta}(\vec{f^{''}})(x)|_r>\alpha \}  \big)\le \omega\left(\{x\notin\Omega: |M^\Delta_{\varphi,\eta}(\overline{f})(x)|_r>\alpha \}  \right)
\\
&\leq \omega\left(\{x\in\mathbb R^n: |M^\Delta_{\varphi,\eta}(\overline{f})(x)|_r>\alpha \} \right)\lesssim \alpha^{-r}\|\overline{f} \|^r_{L^r_\omega(\ell^r,\mathbb R^n)} \lesssim \omega(\Omega),
\end{align*}
which leads to
\begin{align}\label{omega_f"}
&\omega\big(\{x\in\mathbb{R}^n: |M^\Delta_{\varphi,2\eta}(\vec{f^{''}})(x)|_r>\alpha \}  \big)\notag\\
&\le \omega(\Omega)+\omega\big(\{x\notin\Omega: |M^\Delta_{\varphi,2\eta}(\vec{f^{''}})(x)|_r>\alpha \}\big)\lesssim \omega(\Omega).
\end{align}
Besides that, by using (\ref{ineq-fr-alpha}), the H\"{o}lder inequality and $\omega\in A_p^{\Delta,\eta}(\varphi)$, we get
\begin{align*}
&\omega(Q_j)\le \alpha^{-p}\Big(\frac{1}{\varphi(|Q_j|)^{\eta}|Q_j|}\int_{Q_j}|\vec{f}(x)|_rdx \Big)^p.\int_{Q_j}\omega(x)dx\\
&\le \alpha^{-p}\Big(\frac{1}{\varphi(|Q_j|)^{\eta}|Q_j|} \Big)^p \Big(\int_{Q_j} |\vec{f}(x)|_r^p\omega(x)dx \Big) \Big(\int_{Q_j}\omega^{-\frac{p'}{p}}(x)dx \Big)^{\frac{p}{p'}}.\int_{Q_j}\omega(x)dx
\\
& \le \alpha^{-p}\Big( \int_{Q_j} |\vec{f}(x)|^p_r\omega(x)dx\Big) \Big(\frac{1}{\varphi(|Q_j|)^{\eta}|Q_j|} \int_{Q_j}\omega(x)dx \Big) \Big(\frac{1}{\varphi(|Q_j|)^{\eta}|Q_j|} \int_{Q_j}\omega(x)^{-\frac{1}{p-1}}dx  \Big)^{p-1}\\
&\lesssim \alpha^{-p}\Big( \int_{Q_j} |\vec{f}(x)|^p_r\omega(x)dx\Big), \text{for all} ~j\in\mathbb{N}.
\end{align*}
From the above inequality, we infer
\begin{align*}
\omega(\Omega)&=\sum\limits_{j=1}^\infty\omega(Q_j)\lesssim \alpha^{-p} \sum\limits_{j=1}^\infty \int_{Q_j} |\vec{f}(x)|^p_r\omega(x)dx=\alpha^{-p}\int_\Omega |\vec{f}(x)|^p_r\omega(x)dx
\\
&\le \alpha^{-p}\int_{\mathbb{R}^n}|\vec{f}(x)|^p_r\omega(x)dx.
\end{align*}
As an application, by (\ref{omega_f"}), the proof for the inequality (\ref{esti-omega2}) is finished. Finally, since $\vec S$ is dense in $L^p_\omega(\ell^r, \mathbb R^n)$ (see in \cite{BP1961}), the proof of the theorem is ended.
\end{proof}
As a consequence, by combining Theorem \ref{max-Delta-2eta}, Corollary \ref{Cor2.3-WZC2017} and making in the same way as Theorem \ref{max-phi}, we also obtain a necessary condition and a sufficient condition for  the class of  $A_p^{\Delta,\eta}(\varphi)$ weights. More precisely, the following is true.
\begin{theorem}
Let $1 < p < r < \infty$ and $\eta >0$. The following statements are true:
\begin{enumerate}
\item[(i)]	If $\omega\in A_p^{\Delta,\eta}(\varphi)$, then  $M^\Delta_{\varphi,2\eta}$  is a bounded operator from $L^p_\omega(\ell^r,\mathbb{R}^n)$ to $L^{p,\infty}_\omega(\ell^r,\mathbb{R}^n)$, $M^{p}_{q,\omega}(\ell^r,\mathbb{R}^n)$ and $WM^{p}_{q,\omega}(\ell^r,\mathbb{R}^n)$, respectively.
\item [(ii)] If $\omega\notin A_p^{\Delta,\eta}(\varphi)$, then  $M^\Delta_{\varphi,\eta}$  is not a bounded operator from $L^p_\omega(\ell^r,\mathbb{R}^n)$ to $WM^{p}_{q,\omega}(\ell^r,\mathbb{R}^n)$.
\end{enumerate}
\end{theorem}
\section{The results about the boundedness of sublinear operators generated by singular integrals and its commutators}\label{section4}

Let us recall that the two weighted Morrey space $\mathcal B^{p,\kappa}_{\omega_1,\omega_2}(\ell^r,\mathbb R^n)$ with vector-valued functions is defined as the set of all sequences of measurable functions $\vec{f}=\{f_k\}$ such that
\begin{align*}
\|\vec{f} \|_{\mathcal B^{p,\kappa}_{\omega_1,\omega_2}(\ell^r,\mathbb R^n)}=\||\vec{f}(\cdot)|_r \|_{\mathcal B^{p,\kappa}_{\omega_1,\omega_2}(\mathbb R^n)}<\infty.
\end{align*}

It is not difficult to show that  $\mathcal B^{p,\kappa}_{\omega_1,\omega_2}(\ell^r,\mathbb R^n)$ is a Banach space. Our first main result in this section is to give the boundedness of maximal singular integral operators with the kernels proposed by Anderson and John on the space $\mathcal B^{p,\kappa}_{\omega_1,\omega_2}(\ell^r,\mathbb R^n)$. More precisely, we have the following useful result.
\begin{theorem}
Let $1< r<\infty$, $1<p<\infty$, $\omega_1\in A(p,1)$, $\omega_2\in A_p$, $\delta\in (1,r_{\omega_2})$ and $0<\kappa<\frac{\delta-1}{\delta p}$. Then, $T^*$ is a bounded operator on $\mathcal B^{p,\kappa}_{\omega_1,\omega_2}(\ell^r,\mathbb R^n)$.
\end{theorem}
\begin{proof}
Let us choose any $\vec f\in \mathcal B^{p,\kappa}_{\omega_1,\omega_2}(\ell^r,\mathbb R^n)$ and ball $B_R(x_0):= B$. Next, we compose $\vec{f}=\vec{f}_1+\vec{f}_2$, where $\vec{f}_1=\{f_{1,k}\}$ such that $f_{1,k}(x)=f_k(x)\chi_{2B}(x)$. This implies that
\begin{align}\label{J12}
&\frac{1}{\omega_1(B)^{\kappa}}\int_B |T^*(\vec{f})(x)|_r^p\omega_2(x)dx\le \frac{1}{\omega_1(B)^{\kappa}}\int_B |T^*(\vec{f}_1)(x)|_r^p\omega_2(x)dx +\notag\\
&+ \frac{1}{\omega_1(B)^{\kappa}}\int_B |T^*(\vec{f}_2)(x)|_r^p\omega_2(x)dx:=J_1+J_2.
\end{align}
By Theorem \ref{Theo-AJ1981} and Lemma \ref{Lem2.8-CHK1982}, we have
\begin{align}\label{J1}
J_1&\le \frac{1}{\omega_1(B)^{\kappa}}\int_{\mathbb{R}^n} |T^*(\vec{f}_1)(x)|_r^p\omega_2(x)dx\lesssim \frac{1}{\omega_1(B)^{\kappa}}\int_{2B} |\vec{f}(x)|_r^p\omega_2(x)dx\notag\\
&\le \frac{\omega_1(2B)^{\kappa}}{\omega_1(B)^{\kappa}}\|\vec{f}\|^p_{\mathcal{B}^{p,\kappa}_{\omega_1,\omega_2}(\ell^r,\mathbb{R}^n)}\lesssim \|\vec{f}\|^p_{\mathcal{B}^{p,\kappa}_{\omega_1,\omega_2}(\ell^r,\mathbb{R}^n)}.
\end{align}
Now, for $x \in B$ and $y \in (2B)^c$, it is clear to see that $2R\leq |x_0 - y| \le 2|x - y|$. From this, we get
$$
|T^*(f_{2,k})(x)|\leq \int_{(2B)^c}\frac{A}{|x-y|^n}|f_k(y)|dy\lesssim \int_{(2B)^c}\frac{1}{|x_0-y|^n}|f_k(y)|dy,$$  for all $k\in\mathbb N$. Hence, by using the Minkowski inequality and the H\"{o}lder inequality and assuming $\omega_2\in A_p$, we obtain
\begin{align}
&|T^*(\vec f_2)(x)|_r\lesssim \int_{(2B)^c}\frac{|\vec f(y)|_r}{|x_0-y|^n}dy=\sum\limits_{j=1}^{\infty}\,\int_{2^{j}R\leq |x_0-y|<2^{j+1}R}\frac{|\vec f(y)|_r}{|x_0-y|^n}dy\nonumber
\\
&\lesssim \sum\limits_{j=1}^{\infty}\frac{1}{|2^jB|}\int_{2^{j+1}B}|\vec f(y)|_rdy\leq \sum\limits_{j=1}^{\infty}\frac{1}{|2^jB|}\Big(\int_{2^{j+1}B}|\vec f(y)|_r^p\omega_2(y)dy\Big)^{\frac{1}{p}}\Big(\int_{2^{j+1}B}\omega_2(y)^{1-p'}dy\Big)^{\frac{p-1}{p}}\nonumber
\\
&\lesssim \sum\limits_{j=1}^{\infty}\frac{1}{|2^jB|}\Big(\int_{2^{j+1}B}|\vec f(y)|_r^p\omega_2(y)dy\Big)^{\frac{1}{p}}\frac{|2^{j+1}B|}{\omega_2(2^{j+1}B)^{\frac{1}{p}}}\lesssim \|\vec f\|_{\mathcal B^{p,\kappa}_{\omega_1,\omega_2}(\ell^r,\mathbb R^n)}\sum\limits_{j=1}^{\infty}\frac{\omega_1(2^{j+1}B)^{\frac{\kappa}{p}}}{\omega_2(2^{j+1}B)^{\frac{1}{p}}}.\nonumber
\end{align}
Thus, 
\begin{align}
J_2&\lesssim \frac{\omega_2(B)}{\omega_1(B)^{\kappa}} \|\vec f\|^p_{\mathcal B^{p,\kappa}_{\omega_1,\omega_2}(\ell^r,\mathbb R^n)}\Big(\sum\limits_{j=1}^{\infty}\frac{\omega_1(2^{j+1}B)^{\frac{\kappa}{p}}}{\omega_2(2^{j+1}B)^{\frac{1}{p}}}\Big)^{p}=\|\vec f\|^p_{\mathcal B^{p,\kappa}_{\omega_1,\omega_2}(\ell^r,\mathbb R^n)}. \mathcal K^p,\nonumber
\end{align}
where 
$$
\mathcal K= \sum\limits_{j=1}^{\infty}\frac{\omega_1(2^{j+1}B)^{\frac{\kappa}{p}}}{\omega_1(B)^{\frac{\kappa}{p}}}.\frac{\omega_2(B)^{\frac{1}{p}}}{\omega_2(2^{j+1}B)^{\frac{1}{p}}}.
$$
Next, by applying Lemma \ref{Lem2.8-CHK1982}, we have
$
\big(\frac{\omega_1(2^{j+1}B)}{\omega_1(B)}\big)^{\frac{\kappa}{p}}\lesssim \big(\frac{|2^{j+1}B|}{|B|}\big)^{\kappa}\lesssim 2^{{(j+1)n\kappa}}.
$
On the other hand, by using Proposition \ref{rever-Holder}, we infer
$$
\Big(\frac{\omega_2(B)}{\omega_2(2^{j+1}B)}\Big)^{\frac{1}{p}}\lesssim \Big(\frac{|B|}{|2^{j+1}B|}\Big)^{\frac{(\delta-1)}{\delta.p}}\lesssim 2^{\frac{-(j+1)n(\delta-1)}{\delta p}}.
$$
Hence, by $\kappa<\frac{\delta-1}{\delta p}$, one has
$
\mathcal K\lesssim \sum\limits_{j=1}^{\infty} 2^{(j+1)n(\kappa-\frac{\delta-1}{\delta p})}<\infty.
$
Thus,
$$
J_2\lesssim \|\vec f\|^p_{\mathcal B^{p,\kappa}_{\omega_1,\omega_2}(\ell^r,\mathbb R^n)}.
$$
Combining this with (\ref{J12}) and (\ref{J1}) above, we obtain
$$
\|T^*(\vec f)\|_{\mathcal B^{p,\kappa}_{\omega_1,\omega_2}(\ell^r, \mathbb R^n)}\lesssim \|\vec f\|_{\mathcal B^{p,\kappa}_{\omega_1,\omega_2}(\ell^r, \mathbb R^n)},\,\textit{\rm for all }\, \vec f\in \mathcal B^{p,\kappa}_{\omega_1,\omega_2}(\ell^r, \mathbb R^n),
$$
which implies the proof of the theorem is finished.
\end{proof}
\vskip 5pt
Our second main result in this section is to establish the boundedness of sublinear operators  generated by strongly singular operators on certain weighted Morrey spaces. As an application, we obtain the boundedness of some strongly singular integral operators on the weighted Morrey spaces. 

Let us recall the definition of the weighted central Morrey spaces. Let $1 \le q < \infty, 0 < \kappa < 1$ and  $\omega$  be a weighted function. Then the weighted central Morrey spaces is defined as the set of all functions in $L^q_{{\rm loc}}(\mathbb{R}^n)$ such that
\begin{align*}
\|f\|_{\mathcal{{\mathop B\limits^.}}^{q,\kappa}(\omega, \mathbb{R}^n)}=\sup\limits_{R>0} \Big(\frac{1}{\omega(B_R(0))^{\kappa}}\int_{B_R(0)}|f(x)|^qdx \Big)^{\frac{1}{q}}<\infty.
\end{align*}
It is evident that $\mathcal{{\mathop B\limits^.}}^{q,\kappa}(\omega, \mathbb{R}^n)$ is a Banach space. We denote by $\mathfrak{\mathop B\limits^.}^{q,\kappa}(\omega, \mathbb{R}^n)$ the closure of $L^{q}(\mathbb R^n)\cap\mathcal{{\mathop B\limits^.}}^{q,\kappa}(\omega, \mathbb{R}^n)$ with respect to the norm in $\mathcal{{\mathop B\limits^.}}^{q,\kappa}(\omega, \mathbb{R}^n)$.

We also recall that the central weighted local Morrey spaces $\mathcal{\mathop B\limits^.}^{q,\kappa}_{\rm loc}(\omega,\mathbb{R}^n)$ as  the set of all functions in $L^q_{{\rm loc}}(\mathbb{R}^n)$ such that
\begin{align*}
\|f\|_{\mathcal{{\mathop B\limits^.}}^{q,\kappa}_{\rm loc}(\omega, \mathbb{R}^n)}=\sup\limits_{0<R<1} \Big(\frac{1}{\omega(B_R(0))^{\kappa}}\int_{B_R(0)}|f(x)|^qdx \Big)^{\frac{1}{q}}<\infty.
\end{align*}
\begin{theorem}\label{Theo-sublinear1}
Let $1<p<\infty$, $\lambda>0$, $0<\kappa<1$, and $\omega(x)=|x|^{\beta}$ for $-n+\frac{\lambda p}{\kappa}<\beta< \frac{\lambda p +(1-\kappa)n}{\kappa}$ and $\kappa_1 \in (0,\kappa-\frac{\lambda p}{n+\beta}]$. Then, the following is true:

{\rm(i)} If $\mathcal{T}$ extends to a bounded operator on $L^p(\mathbb R^n)$, then $\mathcal{T}$ can also extend to a bounded operator from $\mathfrak{\mathop B\limits^.}^{p,\kappa}(\omega,\mathbb R^n)$ to $\mathcal{{\mathop B\limits^.}}^{p,\kappa_1}_{\rm loc}(\omega, \mathbb{R}^n)$.

{\rm(ii)}  Let $b \in L^{\eta}_{\rm loc}(\mathbb R^n)\cap BMO(\mathbb R^n)$ with $\eta >p'$. If the commutator $[b, \mathcal{T}]$ extends to a bounded operator on $L^p(\mathbb R^n)$, then it can also extend to a bounded operator from $\mathfrak{\mathop B\limits^.}^{p,\kappa}(\omega,\mathbb R^n)$ to $\mathcal{{\mathop B\limits^.}}^{p,\kappa_1}_{\rm loc}(\omega, \mathbb{R}^n)$.
\end{theorem}
\begin{proof} It is sufficient to prove the theorem for all $f\in L^p(\mathbb R^n)\cap\mathcal{{\mathop B\limits^.}}^{p,\kappa}(\omega, \mathbb{R}^n)$.

(i) By fixing a ball $B_R(x_0):=B$ (for $x_0=0$), with $0<R<1$ and decomposing $f=f_1+f_2$, where $f_1= f.\chi_{2B}$, one has
\begin{align}\label{I12-strong}
\frac{1}{\omega(B)^{\kappa_1}}\int_{B}|\mathcal{T}(f)(x)|^pdx &\leq \frac{1}{\omega(B)^{\kappa_1}}\int_{B}|\mathcal{T}(f_1)(x)|^pdx +\nonumber
\\
&\,\,\,+\frac{1}{\omega(B)^{\kappa_1}}\int_{B}|\mathcal{T}(f_2)(x)|^pdx:=I_1+I_2.
\end{align}
To estimate $I_1$, by $\mathcal{T}$ extends to a bounded operator on $L^p(\mathbb R^n)$ and the inequality (\ref{ineq-power}), we have
\begin{align}\label{I1-strong}
I_1&\leq  \frac{1}{\omega(B)^{\kappa_1}}\int_{\mathbb R^n}|\mathcal{T}(f_1)(x)|^pdx \lesssim \frac{1}{\omega(B)^{\kappa_1}}\int_{2B}|f(x)|^pdx \leq \frac{\omega(2B)^{\kappa}}{\omega(B)^{\kappa_1}}\|f\|^p_{\mathcal {\mathop B\limits^.}^{p,\kappa}(\omega,\mathbb R^n)}\nonumber
\\
&\lesssim R^{(n+\beta)(\kappa-\kappa_1)}\|f\|^p_{\mathcal {\mathop B\limits^.}^{p,\kappa}(\omega,\mathbb R^n)}\leq \|f\|^p_{\mathcal {\mathop B\limits^.}^{p,\kappa}(\omega,\mathbb R^n)}.
\end{align}
On the other hand, by $f_2\in L^{p}(\mathbb R^n)$, one has that $g_m =f.\chi_{(2B)^c\cap {(2mB)}}\to f_2$ in $L^p(\mathbb R^n)$. Thus,  by $\mathcal{T}$  bounded on $L^p(\mathbb R^n)$ again, there exists a subsequence $(\mathcal T(g_{m_k}))$-denoted by $(\mathcal T(g_m))$ such that $\mathcal T(g_m)\to \mathcal T(f_2)$ a.e on $\mathbb R^n$. From this, by having $\mathcal{T}$ still satisfies (\ref{ineq-sub}) on $L^{p}_{\rm {comp}} (\mathbb R^n)$ and letting $x\in B$ with $m$ large enough, we obtain
\begin{align}\label{Nakai}
|{\mathcal{T}}{(f_2)}(x)|=\mathop{\rm lim}\limits_{m\to \infty} |{\mathcal{T}}{(g_m)}(x)| \lesssim \mathop{\rm lim}\limits_{m\to \infty} \int_{\mathbb R^n}\frac{|g_m(y)|}{|x-y|^{n+\lambda}}dy= \int_{(2B)^c}\frac{|f(y)|}{|x-y|^{n+\lambda}}dy.
\end{align}
Notice that let $x\in B$ and $y\in (2B)^c$, we have $2R\leq |x_0 - y|\leq 2|x - y|$. This implies that
\begin{align}\label{Tf2}
|\mathcal{T}(f_{2})(x)|&\lesssim \int_{(2B)^c}\frac{1}{|x_0-y|^{n+\lambda}}|f(y)|dy=\sum\limits_{j=1}^{\infty}\,\int_{2^{j}R\leq |x_0-y|<2^{j+1}R}\frac{|f(y)|}{|x_0-y|^{n+\lambda}}dy\nonumber
\\
&\lesssim \sum\limits_{j=1}^{\infty}\frac{1}{|2^jB|^{(1+\frac{\lambda}{n})}}\int_{2^{j+1}B}|f(y)|dy.
\end{align}
From this, by the H\"{o}lder inequality, we deduce
\begin{align}\label{pre-I2}
|\mathcal{T}(f_{2})(x)|&\lesssim \sum\limits_{j=1}^{\infty}\frac{1}{|2^jB|^{(1+\frac{\lambda}{n})}}\Big(\int_{2^{j+1}B}|f(y)|^pdy\Big)^{\frac{1}{p}}|2^{j+1}B|^{\frac{1}{p'}}.\nonumber
\\
&\lesssim \|f\|_{\mathcal {\mathop B\limits^.}^{p,\kappa}(\omega,\mathbb R^n)}\sum\limits_{j=1}^{\infty}\frac{\omega(2^{j+1}B)^{\frac{\kappa}{p}}|2^{j+1}B|^{\frac{1}{p'}}}{|2^{j}B|^{(1+\frac{\lambda}{n})}}.
\end{align}
As a consequence, by (\ref{ineq-power}), we give
\begin{align}\label{I2-strong}
I_2&\lesssim \frac{|B|}{\omega(B)^{\kappa_1}} \|f\|^p_{\mathcal {\mathop B\limits^.}^{p,\kappa}(\omega,\mathbb R^n)}\Big(\sum\limits_{j=1}^{\infty}\frac{\omega(2^{j+1}B)^{\frac{\kappa}{p}}.|2^{j+1}B|^{\frac{1}{p'}}}{|2^{j}B|^{(1+\frac{\lambda}{n})}}\Big)^{p}\nonumber
\\
&\lesssim \|f\|^p_{\mathcal {\mathop B\limits^.}^{p,\kappa}(\omega, \mathbb R^n)}\Big(\sum\limits_{j=1}^{\infty}2^{j(\frac{\kappa(n+\beta)-n}{p}-\lambda)}.R^{\frac{(n+\beta)(\kappa-\kappa_1)}{p}-\lambda}\Big)^p\lesssim \|f\|^p_{\mathcal {\mathop B\limits^.}^{p,\kappa}(\omega,\mathbb R^n)}.
\end{align}
Therefore, by (\ref{I12-strong}) and (\ref{I1-strong}), we immediately have
$$
\|\mathcal{T}(f)\|_{\mathcal {\mathop B\limits^.}^{p,\kappa_1}_{\rm loc}(\omega,\mathbb R^n)}\lesssim \|f\|_{\mathcal {\mathop B\limits^.}^{p,\kappa}(\omega,\mathbb R^n)},\,\textit{\rm for all}\, f\in L^p(\mathbb R^n)\cap\mathcal {\mathop B\limits^.}^{p,\kappa}(\omega,\mathbb R^n), 
$$
which gives that the proof of part (i) is ended.

(ii) As the proof of part (i) above, we also fix a ball $B_R(x_0):=B$ (for $x_0=0$) with $0<R<1$, and write $f=f_1+f_2$ with $f_1= f.\chi_{2B}$. Thus, we get
\begin{align}\label{K12-strong}
&\frac{1}{\omega(B)^{\kappa_1}}\int_{B}|[b,\mathcal{T}](f)(x)|^pdx \leq \frac{1}{\omega(B)^{\kappa_1}}\int_{B}|[b,\mathcal{T}](f_1)(x)|^pdx +\nonumber
\\
&\,\,\,+\frac{1}{\omega(B)^{\kappa_1}}\int_{B}|[b,\mathcal{T}](f_2)(x)|^pdx:=K_1+K_2.
\end{align}
Next, by $[b, \mathcal{T}]$ extends to a bounded operator on $L^p(\mathbb R^n)$ and the relation (\ref{ineq-power}) again, we obtain
\begin{align}\label{K1-strong}
K_1&\leq  \frac{1}{\omega(B)^{\kappa_1}}\int_{\mathbb R^n}|[b,\mathcal{T}](f_1)(x)|^pdx \lesssim \frac{\|b\|^p_{BMO(\mathbb R^n)}}{\omega(B)^{\kappa_1}}\int_{2B}|f(x)|^pdx
\nonumber
\\
&\leq \frac{\omega(2B)^{\kappa}.\|b\|^p_{BMO(\mathbb R^n)}}{\omega(B)^{\kappa_1}}\|f\|^p_{\mathcal {\mathop B\limits^.}^{p,\kappa}(\omega,\mathbb R^n)}\lesssim \|b\|^p_{BMO(\mathbb R^n)}.\|f\|^p_{\mathcal {\mathop B\limits^.}^{p,\kappa}(\omega,\mathbb R^n)}.
\end{align}
Next, by $b\in L^{\eta}_{\rm loc}(\mathbb R^n)$ with $\eta> p'$ and the inequality (\ref{ineq-sub-com}), we get
\begin{align}
\left|[b,\mathcal{T}](g)(x)\right|\lesssim \int_{\mathbb R^n}\frac{|g(y)|.|b(x)-b(y)|}{|x-y|^{n+\lambda}}dy, \text{\; a.e \;} x\not\in \text{supp}{(g)},\forall \, g\in L^{p}_{\rm comp}(\mathbb R^n).\nonumber
\end{align}
Thus, by estimating as (\ref{Nakai}) above and letting $x\in B$ and $y\in (2B)^c$, we have
\begin{align}
|[b,\mathcal{T}](f_{2})(x)|&\lesssim \int_{(2B)^c}\frac{1}{|x-y|^{n+\lambda}}|f(y)|.|b(x)-b(y)|dy
\nonumber
\\
&\leq \int_{(2B)^c}\frac{1}{|x_0-y|^{n+\lambda}}|f(y)|.|b(x)-b(y)|dy\nonumber
\\
&\leq \Big(\int_{(2B)^c}\frac{|f(y)|}{|x_0-y|^{n+\lambda}}dy\Big)|b(x)-b_B|+\int_{(2B)^c}\frac{|f(y)|.|b_B-b(y)|}{|x_0-y|^{n+\lambda}}dy.\nonumber
\end{align}
This leads to that
\begin{align}\label{K21-K22}
K_2&\lesssim \frac{1}{\omega(B)^{\kappa_1}}\Big(\int_{(2B)^c}\frac{|f(y)|}{|x_0-y|^{n+\lambda}}dy\Big)^p.\Big(\int_{B}|b(x)-b_B|^pdx\Big)+\nonumber
\\
&\,\,\,+\frac{|B|}{\omega(B)^{\kappa_1}}\Big(\int_{(2B)^c}\frac{|f(y)|.|b_B-b(y)|}{|x_0-y|^{n+\lambda}}dy\Big)^p :=K_{2,1}+K_{2,2}.
\end{align}
For the term $K_{2,1}$, by using (\ref{Tf2}),   (\ref{pre-I2}), (\ref{I2-strong}) and Lemma \ref{BMO-Lemma}, we infer
\begin{align}\label{K21-strong}
K_{2,1}&\lesssim \frac{|B|}{\omega(B)^{\kappa_1}}\|f\|^p_{\mathcal {\mathop B\limits^.}^{p,\kappa}(\omega,\mathbb R^n)}\Big(\sum\limits_{j=1}^{\infty}\frac{\omega(2^{j+1}B)^{\frac{\kappa}{p}}.|2^{j+1}B|^{\frac{1}{p'}}}{|2^{j}B|^{(1+\frac{\lambda}{n})}}\Big)^{p}\Big(\frac{1}{|B|}\int_{B}|b(x)-b_B|^pdx\Big)\nonumber
\\
&\lesssim \|f\|^p_{\mathcal {\mathop B\limits^.}^{p,\kappa}(\omega,\mathbb R^n)}.\|b\|^p_{BMO(\mathbb R^n)}.
\end{align}
For the term $K_{2,2}$, by the H\"{o}lder inequality, we have
\begin{align}
K_{2,2}&\lesssim \frac{|B|}{\omega(B)^{\kappa_1}}\Big(\sum\limits_{j=1}^{\infty}\frac{1}{|2^jB|^{(1+\frac{\lambda}{n})}}\int_{2^{j+1}B}|f(y)|.|b_B-b(y)|dy\Big)^p\nonumber
\\
&\leq\frac{|B|}{\omega(B)^{\kappa_1}}\Big(\sum\limits_{j=1}^{\infty}\frac{1}{|2^jB|^{(1+\frac{\lambda}{n})}}\Big(\int_{2^{j+1}B}|f(y)|^pdy\Big)^{\frac{1}{p}}.\Big(\int_{2^{j+1}B}|b_B-b(y)|^{p'}dy\Big)^{\frac{1}{p'}}\Big)^p\nonumber
\\
&\lesssim \frac{|B|}{\omega(B)^{\kappa_1}}\Big(\sum\limits_{j=1}^{\infty}\frac{\omega(2^{j+1}B)^{\frac{\kappa}{p}}}{|2^jB|^{(1+\frac{\lambda}{n})}}(L_{1,i}+L_{2,i})\Big)^p\|f\|^p_{\mathcal {\mathop B\limits^.}^{p,\kappa}(\omega,\mathbb R^n)},\nonumber
\end{align}
where $L_{1,i}=\Big(\int_{2^{j+1}B}|b(y)-b_{2^{j+1}B}|^{p'}dy\Big)^{\frac{1}{p'}}$ and $L_{2,i}=\Big(\int_{2^{j+1}B}|b_B-b_{2^{j+1}B}|^{p'}dy\Big)^{\frac{1}{p'}}$.
On the other hand, by Lemma \ref{BMO-Lemma} and Proposition \ref{Pro-T1986}, we also get
$$
L_{1,i}\leq \|b\|_{BMO(\mathbb R^n)}.|2^{j+1}B|^{\frac{1}{p'}}
$$
and
$$
L_{2,i}\leq \Big(\int_{2^{j+1}B}\Big(2^n(j+1)\|b\|_{BMO(\mathbb R^n)}\Big)^{p'}dy\Big)^{\frac{1}{p'}}\leq 2^n(j+1).\|b\|_{BMO(\mathbb R^n)}.|2^{j+1}B|^{\frac{1}{p'}}.
$$
Thus, by estimating as (\ref{I2-strong}) above, we immediately have
\begin{align}
K_{2,2}&\lesssim \frac{|B|}{\omega(B)^{\kappa_1}}\Big(\sum\limits_{j=1}^{\infty}\frac{(j+2).\omega(2^{j+1}B)^{\frac{\kappa}{p}}.|2^{j+1}B|^{\frac{1}{p'}}}{|2^jB|^{(1+\frac{\lambda}{n})}}\Big)^p\|f\|^p_{\mathcal B^{p,\kappa}(\omega,\mathbb R^n)}.\|b\|^p_{BMO(\mathbb R^n)}\nonumber
\\
&\lesssim \Big(\sum\limits_{j=1}^{\infty}(j+2)2^{j(\frac{\kappa(n+\beta)-n}{p}-\lambda)}\Big)^p.\|f\|^p_{\mathcal B^{p,\kappa}(\omega,\mathbb R^n)}.\|b\|^p_{BMO(\mathbb R^n)}\nonumber
\\
&\lesssim \|f\|^p_{\mathcal {\mathop B\limits^.}^{p,\kappa}(\omega,\mathbb R^n)}.\|b\|^p_{BMO(\mathbb R^n)}.\nonumber 
\end{align}
From the above estimation, by (\ref{K12-strong})-(\ref{K21-strong}), we confirm
$$
\|[b,\mathcal{T}](f)\|_{\mathcal {\mathop B\limits^.}^{p,\kappa_1}_{\rm loc}(\omega,\mathbb R^n)}\lesssim \|b\|_{BMO(\mathbb R^n)}.\|f\|_{\mathcal {\mathop B\limits^.}^{p,\kappa}(\omega,\mathbb R^n)},\,\textit{\rm for all}\, f\in L^p(\mathbb R^n)\cap \mathcal {\mathop B\limits^.}^{p,\kappa}(\omega,\mathbb R^n).
$$
Therefore, the proof of this theorem is completed.
\end{proof}
Now, let us give some applications of Theorem \ref{Theo-sublinear1}. Note that Hirschman \cite{Hirschman1959},  Wainger \cite{Wainger1965}, Cho and Yang \cite{CY2010} studied the strongly singular convolution operators in the context of $L^p(\mathbb R^n)$ spaces defined as follows.
\begin{definition} 
Let $0 < s < \infty$ and $0 < \lambda < \frac{ns}{2}$. The strongly singular integral operator $T^{s,\lambda}$ is defined	by
\begin{align*}
T^{s,\lambda}(f)(x)=p.v.\int_{\mathbb{R}^n}\frac{e^{i|x-y|^{-s}}}{|x-y|^{n+\lambda}}\chi_{\{|x-y|<1 \}}f(y)dy.
\end{align*}
\end{definition}
\begin{theorem}\label{Theo-strong}
{\rm (see in \cite{Fefferman1970, Hirschman1959, Wainger1965})} Let $0<s<\infty$, $1<p<\infty$, $0<\lambda<\frac{ns}{2}$, $|\frac{1}{p}-\frac{1}{2}|<\frac{1}{2}-\frac{\lambda}{ns}$. Then $T^{s,\lambda}$ extends to a bounded operator from $L^p(\mathbb R^n)$ to itself.
\end{theorem}
\begin{definition} 
Let $0 < \zeta, s,\lambda < \infty$, and $k$ be an integer with $k\geq 2$. The strongly singular integral operator $T_{\zeta,s,\lambda}$ is defined	by
\begin{align*}
T_{\zeta,s, \lambda}(f)(x)=p.v.\int_{\mathbb{R}}\frac{e^{i\{\zeta.(x-y)^k+|x-y|^{-s}\}}}{(x-y)|x-y|^{\lambda}}f(y)dy.
\end{align*}
\end{definition}
\begin{theorem}\label{Theo-strong-CY2010}
{\rm (see in \cite{CY2010})} Let $0 < \zeta, s, \lambda < \infty$, $k\in\mathbb N$ with $k\geq 2$ and $s\geq 2\lambda$. Then $T_{\zeta,s,\lambda}$ extends to a bounded operator from $L^2(\mathbb R)$ to itself.
\end{theorem}
On the other hand,  Li and Lu \cite{LL2006} also studied the Coifman-Rochberg-Weiss type commutator of strongly singular integral operator defined as follows
\begin{definition}
Let $0 < s < \infty$ and $0 < \lambda < \frac{ns}{2}$. The Coifman-Rochberg-Weiss type commutator of strongly singular integral operator is defined by
\begin{equation}\label{commuatator-strong}
[b,T^{s,\lambda}](f)(x)=p.v.\int_{\mathbb{R}^n}\frac{e^{i|x-y|^{-s}}}{|x-y|^{n+\lambda}}\chi_{\{|x-y|<1 \}}\big(b(x)-b(y)\big)f(y)dy,
\end{equation}
where $b$ is locally integrable functions on $\mathbb R^n$.
\end{definition}
Moreover, Li and Lu \cite{LL2006} proved the following interesting result.
\begin{theorem}\label{Theo1.1-LL2006}
{\rm (Theorem 1.1 \cite{LL2006})} Let $0 < s < \infty$, $1<p<\infty$, $0<\lambda<\frac{ns}{2}$, $|\frac{1}{p}-\frac{1}{2}|<\frac{1}{2}-\frac{\lambda}{ns}$  and $b\in BMO(\mathbb R^n)$. Then the commutator $[b, T^{s,\lambda}]$ extends to  a bounded operator on $L^p(\mathbb R^n)$.
\end{theorem}
From Theorem \ref{Theo-sublinear1}, Theorem \ref{Theo-strong}, Theorem \ref{Theo-strong-CY2010} and Theorem \ref{Theo1.1-LL2006}, we obtain the useful results as follows.
\begin{corollary}\label{Theo-strong1}
Let $0<s<\infty$, $1<p<\infty$, $0<\lambda<\frac{ns}{2}$, $0<\kappa<1$, $|\frac{1}{p}-\frac{1}{2}|<\frac{1}{2}-\frac{\lambda}{ns}$, $-n+\frac{\lambda p}{\kappa}<
\beta< \frac{\lambda p +(1-\kappa)n}{\kappa}$, $\omega(x)=|x|^{\beta}$ and $\kappa_1\in (0,\kappa-\frac{\lambda p}{n+\beta}]$. Let $b\in L^{\eta}_{\rm loc}(\mathbb R^n)\cap BMO(\mathbb R^n)$ with $\eta>p'$.
Then $T^{s,\lambda}$ and $[b, T^{s,\lambda}]$ extend to  bounded operators from $\mathfrak {\mathop B\limits^.}^{p,\kappa}(\omega,\mathbb R^n)$ to $\mathcal {\mathop B\limits^.}^{p,\kappa_1}_{\rm loc}(\omega,\mathbb R^n)$.
\end{corollary}
\begin{corollary}\label{Coro-CY-new}
Let $0<\zeta,\lambda,s<\infty$, $k\in\mathbb N$ with $k\geq 2$, $s\geq 2\lambda$, $0<\kappa<1$, $-1+\frac{2\lambda}{\kappa}<
\beta< \frac{2\lambda +(1-\kappa)}{\kappa}$, $\omega(x)=|x|^{\beta}$ and $\kappa_1\in (0,\kappa-\frac{2\lambda}{1+\beta}]$. Then $T_{\zeta, s,\lambda}$ extends to a bounded operator from $\mathfrak {\mathop B\limits^.}^{2,\kappa}(\omega,\mathbb R)$ to $\mathcal {\mathop B\limits^.}^{2,\kappa_1}_{\rm loc}(\omega,\mathbb R)$.
\end{corollary}
Remark that in the special case when the weight function in Theorem \ref{Theo-sublinear1}, Corollary \ref{Theo-strong1} and Corollary \ref{Coro-CY-new} is a constant function, then we can remove the central condition in the spaces, that is, we may replace 
$\mathfrak {\mathop B\limits^.}^{p,\kappa}(\mathbb R^n)$ and $\mathcal {\mathop B\limits^.}^{p,\kappa_1}_{\rm loc}(\mathbb R^n)$ by $\mathfrak {M}^{p,\kappa}(\mathbb R^n)$ and $\mathcal { B}^{p,\kappa_1}_{\rm loc}(\mathbb R^n)$, respectively. Here ${\mathfrak{M} }^{p,\kappa}(\mathbb R^n)$ is the closure of $L^{p}(\mathbb R^n)\cap {\mathcal B}^{p,\kappa}(\mathbb R^n)$ in the space ${\mathcal B}^{p,\kappa}( \mathbb R^n)$.
\vskip 5pt
Finally, we give the boundedness of sublinear operators in the setting when the weighted function is in the class of Muckenhoupt weights. It is worth pointing out that when $\omega_1=\omega_2=\omega$, then the space $C^\infty_0(\mathbb R^n)$ is contained in ${\mathcal B}^{q,\kappa}_\omega(\mathbb R^n)$. The space  ${\mathfrak{B} }^{q,\kappa}_\omega(\mathbb R^n)$ is denoted as the closure of $L^{q}(\mathbb R^n)\cap {\mathcal B}^{q,\kappa}_\omega(\mathbb R^n) $ in the space ${\mathcal B}^{q,\kappa}_\omega(\mathbb R^n)$.
\begin{theorem}\label{Theo-Sublinear2}
Let $1<p<\infty$, $0<\kappa<1$, and $1\leq p^*,\zeta<\infty$, $\omega\in A_{\zeta}$ with the finite critical index $r_\omega$ for the reverse H\"{o}lder. Assume that $p > p^{*}\zeta {r^{'}_\omega}, \delta\in (1,r_\omega)$ and $\kappa^*=\frac{p^*(\kappa-1)}{p}+1$. Then, if $\mathcal{T}$ extends to a bounded operator on $L^p(\mathbb R^n)$, then $\mathcal{T}$ can also extend to a bounded operator from ${\mathfrak{B} }^{p,\kappa}_\omega(\mathbb R^n)$ to ${B}^{p^*,\kappa^*}_\omega(\mathbb R^n)$.

\end{theorem}

\begin{proof}
Let us fix $f\in L^{p}(\mathbb R^n)\cap {\mathcal B}^{p,\kappa}_\omega(\mathbb R^n)$ and a ball $B_R(x_0):=B$ with $R\geq 1$. From assume that $p > p^*\zeta r^{'}_\omega $, one has $r\in(1, r_\omega)$ satisfying
$p = \zeta p^* r'$. Hence, by the H\"{o}lder inequality and the reverse H\"{o}lder condition, we lead to
\begin{align}\label{T_p*}
\Big(\int_{B}|\mathcal{T} (f)(x)|^{p*}\omega(x)dx\Big)^{\frac{1}{p^*}}&\leq \Big(\int_{B}|\mathcal{T} (f)(x)|^{\frac{p}{\zeta}}dx\Big)^{\frac{\zeta}{p}}.\Big(\int_B \omega(x)^r dx\Big)^{\frac{1}{rp*}}\nonumber
\\
&\lesssim \Big(\int_{B}|\mathcal{T}(f)(x)|^{\frac{p}{\zeta}}dx\Big)^{\frac{\zeta}{p}}\omega(B)^{\frac{1}{p*}}.|B|^{\frac{-\zeta}{p}}.
\end{align}
Next, we decompose $f=f_1+f_2$ where $f_1=f.\chi_{2B}$. Thus,
\begin{align}\label{A12}
\Big(\int_{B}|\mathcal{T}(f)(x)|^{\frac{p}{\zeta}}dx\Big)^{\frac{\zeta}{p}}&\lesssim  \Big(\int_{B}|\mathcal{T} (f_1)(x)|^{\frac{p}{\zeta}}dx\Big)^{\frac{\zeta}{p}}+ \Big(\int_{B}|\mathcal{T}(f_2)(x)|^{\frac{p}{\zeta}}dx\Big)^{\frac{\zeta}{p}}\nonumber
\\
&:=A_1+A_2.
\end{align}
From  assuming that $\mathcal{T}$ extends to a bounded operator on $L^p(\mathbb R^n)$ and using Proposition   \ref{pro2.4DFan}, we get
\begin{align}\label{A1}
A_1&\leq \Big(\int_{\mathbb R^n}|\mathcal{T}( f_1)(x)|^{\frac{p}{\zeta}}dx\Big)^{\frac{\zeta}{p}}\lesssim \Big(\int_{2B}|f(x)|^{\frac{p}{\zeta}}dx\Big)^{\frac{\zeta}{p}}.\nonumber
\\
&\lesssim \Big(\int_{2B}|f(x)|^{p}\omega(x)dx\Big)^{\frac{1}{p}}\omega(2B)^{\frac{-1}{p}}|2B|^{\frac{\zeta}{p}}\lesssim \|f\|_{\mathcal B^{p,\kappa}_{\omega}(\mathbb R^n)}.\omega(2B)^{\frac{(\kappa-1)}{p}}.|B|^{\frac{\zeta}{p}}.
\end{align}
Next, let us give $x\in B$ and $y\in (2B)^c$. By applying the relation (\ref{Nakai}) above and estimating as (\ref{Tf2}) and (\ref{A1}), we obtain
\begin{align}
&|\mathcal{T}(f_2)(x)|\nonumber
\\
&\lesssim \sum\limits_{j=1}^{\infty}\frac{1}{|2^jB|^{(1+\frac{\lambda}{n})}}\int_{2^{j+1}B}|f(y)|dy\leq \sum\limits_{j=1}^{\infty}\frac{1}{|2^jB|^{(1+\frac{\lambda}{n})}}\Big(\int_{2^{j+1}B}|f(y)|^{\frac{p}{\zeta}}dy\Big)^{\frac{\zeta}{p}}|2^{j+1}B|^{1-\frac{\zeta}{p}} \nonumber
\\
&\leq \sum\limits_{j=1}^{\infty}\frac{1}{|2^jB|^{(1+\frac{\lambda}{n})}}\|f\|_{\mathcal B^{p,\kappa}_{\omega}(\mathbb R^n)}.\omega(2^{j+1}B)^{\frac{(\kappa-1)}{p}}.|2^{j+1}B|\lesssim \|f\|_{\mathcal B^{p,\kappa}_{\omega}(\mathbb R^n)}\sum\limits_{j=1}^{\infty}\frac{\omega(2^{j+1}B)^{\frac{(\kappa-1)}{p}}}{|2^jB|^{\frac{\lambda}{n}}}.\nonumber
\end{align}
Hence,
\begin{align}\label{A2}
A_2\lesssim \|f\|_{\mathcal B^{p,\kappa}_{\omega}(\mathbb R^n)}\Big(\sum\limits_{j=1}^{\infty}\frac{\omega(2^{j+1}B)^{\frac{(\kappa-1)}{p}}}{|2^jB|^{\frac{\lambda}{n}}}\Big)|B|^{\frac{\zeta}{p}}.
\end{align}
Next, by Proposition \ref{rever-Holder} and $\kappa\in (0,1)$, we deduce
\begin{align}
\Big(\frac{\omega(2^{j+1}B)}{\omega(B)}\Big)^{\frac{(\kappa-1)}{p}}\lesssim \Big(\frac{|2^{j+1}B|}{|B|}\Big)^{\frac{(\kappa-1)(\delta-1)}{p\delta}}\lesssim 2^{\frac{jn(\kappa-1)(\delta-1)}{p\delta}}.\nonumber
\end{align}
From this, by using (\ref{T_p*})-(\ref{A2}), $\kappa^*=\frac{p^*(\kappa-1)}{p}+1$ and $R\geq 1$, we get
\begin{align}
&\Big(\frac{1}{\omega(B)^{\kappa^*}}\int_{B}|\mathcal{T}(f)(x)|^{p*}\omega(x)dx\Big)^{\frac{1}{p^*}}\nonumber
\\
&\lesssim \frac{\|f\|_{\mathcal B^{p,\kappa}_{\omega}(\mathbb R^n)}}{\omega(B)^{\frac{\kappa^*}{p*}}}.\Big(\omega(2B)^{\frac{(\kappa-1)}{p}}+\sum\limits_{j=1}^{\infty}2^{-j\lambda}\omega(2^{j+1}B)^{\frac{(\kappa-1)}{p}}\Big).\omega(B)^{\frac{1}{p*}}\nonumber
\\
&\lesssim \|f\|_{\mathcal B^{p,\kappa}_{\omega}(\mathbb R^n)}.\Big(\sum\limits_{j=0}^{\infty}2^{-j\lambda}\Big(\frac{\omega(2^{j+1}B)}{\omega(B)}\Big)^{\frac{(\kappa-1)}{p}}\Big)
\lesssim \|f\|_{\mathcal B^{p,\kappa}_{\omega}(\mathbb R^n)}.\Big(\sum\limits_{j=0}^{\infty}2^{j(\frac{n(\kappa-1)(\delta-1)}{p\delta}-\lambda)}\Big)\nonumber
\\
&\lesssim \|f\|_{\mathcal B^{p,\kappa}_{\omega}(\mathbb R^n)}.\nonumber
\end{align}
Therefore,
$$
\|\mathcal{T}(f)\|_{ B^{p^*,\kappa^*}_\omega(\mathbb R^n)}\lesssim \|f\|_{\mathcal B^{p,\kappa}_\omega(\mathbb R^n)},\,\textit{\rm for all}\,f\in L^p(\mathbb R^n)\cap {\mathcal B}^{p,\kappa}_\omega(\mathbb R^n).
$$
This implies that theorem is proved.
\end{proof}
By Theorem \ref{Theo-Sublinear2}, we obtain the following interesting corollary.
\begin{corollary}\label{Theo-SIO-strong-2}
Let $s,\lambda,\kappa,p$ as Corollary \ref{Theo-strong1} and $1\leq p^*,\zeta<\infty$, $\omega\in A_{\zeta}$ with the finite critical index $r_\omega$ for the reverse H\"{o}lder. Assume that $p > p^{*}\zeta {r^{'}_\omega}, \delta\in (1,r_\omega)$ and $\kappa^*=\frac{p^*(\kappa-1)}{p}+1$. Then, $T^{s,\lambda}$ can extend to a bounded operator from ${\mathfrak{B} }^{p,\kappa}_\omega(\mathbb R^n)$ to ${B}^{p^*,\kappa^*}_\omega(\mathbb R^n)$.
\end{corollary}

\bibliographystyle{amsplain}

\end{document}